\documentclass{amsart}
\usepackage{latexsym,amscd,amssymb,url}
\usepackage{extarrows}
\usepackage{verbatim}
\usepackage{mathtools}
\usepackage{bbm}
\usepackage{dsfont}
\usepackage{enumitem}
\usepackage{tikz-cd}

\usepackage[all,cmtip]{xy}  
\usepackage{graphicx}
\usepackage{xcolor}
\usepackage{enumitem} 
\definecolor{dblue}{rgb}{0,0,.6}
\usepackage[colorlinks=true, linkcolor=dblue, citecolor=dblue, filecolor = dblue, menucolor = dblue, urlcolor = dblue]{hyperref}

\numberwithin{equation}{section}

\newtheorem{theorem}{Theorem}[section]

\theoremstyle{plain}

\newtheorem{claim}{Claim}

\newtheorem{corollary}[theorem]{Corollary}

\newtheorem{lemma}[theorem]{Lemma}

\newtheorem{proposition}[theorem]{Proposition}
\newtheorem{remark}[theorem]{Remark}

\newcommand{\Z}{\mathbb Z}
\newcommand{\Q}{\mathbb Q}
\newcommand{\Qbar}{\overline{\mathbb Q}}
\newcommand{\C}{\mathbb C}

\newcommand{\R}{\operatorname{R}}

\newcommand{\CP}{\mathbb P}

\newcommand{\A}{A}

\newcommand{\im}{\operatorname{im}}

\newcommand{\Hom}{\operatorname{Hom}}
\newcommand{\Pic}{\operatorname{Pic}}

\newcommand{\Fix}{\operatorname{Fix}}

\newcommand{\id}{\operatorname{id}}

\newcommand{\Spec}{\operatorname{Spec}}

\newcommand{\Gal}{\operatorname{Gal}}
\newcommand{\res}{\operatorname{res}}
\newcommand{\pr}{\operatorname{pr}}
\newcommand{\NS}{\operatorname{NS}}
\newcommand{\Sing}{\operatorname{Sing}}
\newcommand{\Alb}{\operatorname{Alb}}

\newcommand{\codim}{\operatorname{codim}}

\newcommand{\CH}{\operatorname{CH}}
\newcommand{\supp}{\operatorname{supp}}
\newcommand{\sing}{\operatorname{sing}}

\newcommand{\cl}{\operatorname{cl}} 
 
\newcommand{\Frac}{\operatorname{Frac}}
\newcommand{\Char}{\operatorname{char}}

 \newcommand{\Exc}{\operatorname{Exc}}

  \newcommand{\coker}{\operatorname{coker}}

\newcommand{\tors}{\operatorname{tors}}

\newcommand{\alg}{\operatorname{alg}}

\newcommand{\et}{\text{\'et}}
\newcommand{\zar}{\text{Zar}}

\newcommand{\Ab}{\operatorname{Ab}}

\newcommand{\cd}{\operatorname{cd}}

\newcommand{\cont}{\operatorname{cont}}

\newcommand{\dashedlongrightarrow}{\xymatrix@1@=15pt{\ar@{-->}[r]&}}
\renewcommand{\longrightarrow}{\xymatrix@1@=15pt{\ar[r]&}}
\renewcommand{\mapsto}{\xymatrix@1@=15pt{\ar@{|->}[r]&}}
\renewcommand{\twoheadrightarrow}{\xymatrix@1@=15pt{\ar@{->>}[r]&}}
\newcommand{\hooklongrightarrow}{\xymatrix@1@=15pt{\ar@{^(->}[r]&}}
\newcommand{\congpf}{\xymatrix@1@=15pt{\ar[r]^-\sim&}}
\renewcommand{\cong}{\simeq}

\begin{document}

\title[Torsion higher Chow cycles modulo $\ell$]{Torsion higher Chow cycles modulo $\ell$}

\author{Theodosis Alexandrou} 
\address{Institut für Mathematik, Humboldt-Universität zu Berlin, Rudower Chaussee 25, 10099 Berlin, Deutschland.}
\email{theodosis.alexandrou@hu-berlin.de} 

\author{Lin Zhou} 
\address{Institut für Algebraische Geometrie, Leibniz Universität Hannover, Welfengarten 1, 30167 Hannover, Deutschland.}
\email{zhou@math.uni-hannover.de}

\date{\today} 


 \begin{abstract} We study the injectivity property of certain actions of higher Chow groups on refined unramified cohomology. As an application for every $p\geq1$ and for each $d\geq p+4$ and $n\geq2,$ we establish the first examples of smooth complex projective $d$-folds $X$ such that for all $p+3\leq c\leq d-1,$ the higher Chow group $\CH^{c}(X,p)$ contains infinitely many torsion cycles of order $n$ that remain linearly independent modulo $n$. Our bounds for $c$ and $d$ are also optimal. A crucial tool for the proof is morphic cohomology.
 \end{abstract}

\maketitle 

\section{Introduction}
Let $X$ be a smooth complex projective variety. It is known that the Chow group of codimension-$i$ algebraic cycles on $X$ modulo a prime number $\ell$ might well be infinite if $2\leq i\leq \dim X-1$. Schoen \cite{schoen-modn} showed that this phenomenon occurs for the triple product of the Fermat elliptic curve and all prime numbers $\ell\equiv1\mod{3}$. Later, Totaro \cite{totaro-annals} building upon earlier works of Rosenschon--Srinivas \cite{RS} and of Nori \cite{Nori} proved that for a very general principally polarized complex abelian 3-fold $X$, the Chow group $\CH^{2}(X)/\ell$ is infinite for all prime numbers $\ell$. An example of infinite-dimensionality of $\CH^{i}(X)/\ell$ for all prime numbers $\ell$ was also constructed by Diaz \cite{Diaz}. All the aforementioned results strongly rely on a fundamental theorem of Bloch and Esnault \cite{Bloch-Esnault} that provides a criterion for a nullhomologous cycle to give a non-zero element in the second Chow group modulo $\ell$. A variant of Bloch and Esnault's theorem was recently given by Farb--Kisin--Wolfson \cite{fkw}, whose proof uses the prismatic theory of Bhatt--Scholze \cite{prismatic}. Scavia \cite{scavia} applied their result to extend the main result of \cite{schoen-modn} from $\ell\equiv 1\mod{3}$ to all prime numbers $\ell>5$, thus giving the first example of a smooth projective variety $X$ over $\Qbar$ such that $\CH^2(X)/\ell$ is infinite for all but at most finitely many prime numbers $\ell$.


\par With regards to a problem of Schoen \cite{schoen-torsion}, Schreieder \cite{Sch-griffiths} constructed examples of smooth complex projective $5$-folds with infinite $2$-torsion in their third Griffiths groups, proving so that the torsion subgroup of Griffiths groups is in general not finite. The result was recently extended to odd primes by the first named author \cite{Ale-griffiths}. Since the infinitely many $\ell$-torsion classes considered in \cite{Sch-griffiths,Ale-griffiths} are linearly independent modulo $\ell$, one finds that the group $\CH^{3}(X)[\ell]/\ell\CH^{3}(X)[\ell^{2}]$ is also infinite.
\par In the present article, parallel to the above results, we study the structure of Bloch's higher Chow group $\CH^{c}(X,p)$ \cite{bloch-motivic}. For $p=1$, the $K$-theoretic description $\CH^{c}(X,1)\cong H^{c-1}(X_{\zar},\mathcal{K}_{c})$ \cite{quillen} allows one to interpret higher cycles in the form $\sum Y_{i}\otimes f_{i},$ where $Y_{i}\subset X$ are irreducible subvarieties of codimension-$(c-1)$ and where $f_{i}\in \C(Y_{i})^{\ast}$ are rational functions that respect the rule $\sum \text{div}(f_{i})=0$ as a cycle on $X$. It is well-known that the vector space $\CH^{c}(X,1)_{\Q}$ is in general not countable even modulo decomposable cycles, i.e., modulo the subspace $$\CH^{c}(X,1)_{\text{dec},\Q}:=\im( \CH^{1}(X,1)\otimes\CH^{c-1}(X)\longrightarrow\CH^{c}(X,1))\otimes\Q,$$ where one recalls that $\CH^{1}(X,1)=\C^{\ast}$. In particular, producing examples with uncountably many indecomposable higher Chow cycles has been the subject of various works; see \cite{collino,co-fa,gordon-lewis,gordon-lewis-proceedings,lewis}.\par
Interestingly enough, it appears that cases in which the countable groups $\CH^{c}(X,p)[\ell]$ and $\CH^{c}(X,p)/\ell$ are infinite are lacking in the literature. This work is thus dedicated to show the following.

\begin{theorem}\label{thm:main} For each $p\geq1$ and $n\geq2$, there is a smooth complex projective variety $X$ of dimension $p+4$ such that the group $\CH^{p+3}(X,p)[n]/n\CH^{p+3}(X,p)[n^2]$ contains infinitely many cycles of order $n$.
\end{theorem}


\par Considering products $X\times_{\C}\CP^{d-(p+4)}_{\C}$ for any $d\geq p+4$, the projective bundle formula also gives examples in higher dimensions.

\begin{corollary}\label{cor:main} Let $p\geq1$. For each $d\geq p+4$ and $n\geq2,$ there is a smooth complex projective $d$-fold $X$ such that the group $\CH^{c}(X,p)[n]/n\CH^{c}(X,p)[n^2]$ contains infinitely many cycles of order $n$ for all $p+3\leq c\leq d-1$.\end{corollary}

\par We note that the bounds given in the corollary above are optimal. In fact, if either $c<p+2$ or $c\geq\dim X$, then the group $\CH^{c}(X,p)/n$ is always finite; see Lemma \ref{lem:bound_mod_n} \eqref{it:bound_mod_n}. 
In addition, for each $n\geq2$, the $n$-torsion subgroup $\CH^{p+2}(X,p)[n]$ is finite, since $\CH^{p+2}(X,p)_{\tors}\cong N^{1}H^{p+3}(X_{\et},\Q/\Z(p))$; see Theorem \ref{thm:Licthenbaum-Beilinson-conjecture} \eqref{it:L-B'} (or see Lemma \ref{lem:bound_mod_n} \eqref{it:bound_torsion}). The latter for $p=0$ is a theorem of Merkurjev--Suslin \cite{merkurjev-suslin} and for $p\geq1$ is a consequence of the Beilinson--Lichtenbaum conjecture proven by Voevodsky \cite{Voe-milnor,Voevodsky}. 

\par A crucial input for the proof of Theorem \ref{thm:main} is the so-called morphic cohomology $L^{c}H^{p}(X)$ of $X$. That is, the Poincar\'e dual of the Lawson homology $L_{\dim(X)-c}H_{2\dim(X)-p}(X)$; see \cite{lawson-a,lawson-b}. Lawson laid the foundations for a theory based on the homotopy groups of Chow varieties whose properties have been studied and further developed in various works; see \cite{Fri91, FL92, FM94, suslin-voevodsky, Fri00}. Unlike the case of Bloch's higher Chow groups that extend the definition of the Chow group of rational equivalence classes of algebraic cycles on a variety $X$, the precursor of the Lawson homology theory replaces the role of rational equivalence by algebraic equivalence. Namely, if $p=2c$, then $L^{c}H^{p}(X)$ is identified with the Chow group of codimension-$c$ algebraic cycles on $X$ modulo algebraic equivalence, i.e., $L^{c}H^{2c}(X)=\CH^{c}(X)/\sim_{\alg}$. More generally, for an abelian group $A$, the morphic cohomology $L^{c}H^{p}(X,A)$ and Bloch's higher Chow group $\CH^{c}(X,2c-p\ ;A)$ with coefficients in $A$, are related via a comparison map
$$\CH^{c}(X,2c-p\ ;A)\longrightarrow L^{c}H^{p}(X,A),$$
 which becomes an isomorphism if one takes $A=\Z/m$; see Theorem \ref{thm:suslin-voevodsky} \eqref{it:S-V}. In particular, one finds for all integers $c,p\geq0$ and $n\geq2,$ a natural isomorphism 
 \begin{equation}\label{eq:comparison-L-M}
    \CH^c(X,p)[n]/n\CH^c(X,p)[n^2]\cong L^cH^{2c-p}(X)[n]/nL^cH^{2c-p}(X)[n^2];
\end{equation}
see Theorem \ref{thm:suslin-voevodsky} \eqref{it:S-V'}. This allows us to reduce Theorem \ref{thm:main} to a statement about morphic cohomology; see Theorem \ref{thm:main-Lawson}.

\par Our approach towards the proof of Theorem \ref{thm:main} is similar to \cite{Sch-griffiths, Ale-griffiths}. That is, we establish various injectivity theorems about exterior product maps of higher Chow groups that generalize the results \cite[Theorem 6.1]{Sch-griffiths} and \cite[Theorem 0.2]{schoen-products}; see $\S$\ref{sec:inj}. As in \cite{Sch-griffiths}, these results rely on the fact that several groups of algebraic cycles admit a description in terms of refined unramified cohomology \cite{Sch-refined}. Most notably, the kernel of the cycle class map $\bar{\cl}^{p,c}_{X}:\CH^{c}(X,\ 2c-p\ ; \Z/n\Z)\to H^{p}(X_{\et},\Z/n(c))$, which became accessible by recent works of Kok and the second named author, as well as the first named author and Schreieder; see \cite{kok-zhou, alexandrou-schreieder}. The injectivity theorem that is relevant to Theorem \ref{thm:main} has the following consequence. 

\begin{theorem}\label{thm:injectivity-result-intro} Let $n\geq 2$ be an integer. Let $Y$ be a smooth complex projective variety. Then there is a smooth complex projective surface $S$ with $\Pic(S)_{\tors}\cong\Z/n\Z$ such that the exterior product map 
\begin{align}\label{eq:exterior-product}
    L^{2}H^{3}(S)/n\otimes L^{c}H^{2c-p}(Y)\longrightarrow L^{c+2}H^{2c-p+3}(S\times Y)/n,\  [\Gamma]\otimes[z]\mapsto [\Gamma\times z]
\end{align}
is injective for all integers $c,p\geq0$. 
\end{theorem}
The condition $\Pic(S)_{\tors}\cong\Z/n\Z$ in the above theorem implies
\begin{align*}
   \frac{L^{2}H^{3}(S)[n]}{nL^{2}H^{3}(S)[n^{2}]} =L^{2}H^{3}(S)/n\cong\Z/n\Z;
\end{align*}
see Lemma \ref{lem:properties} \eqref{it:Pr-b}.

Thus, Theorem \ref{thm:main} in turn follows from Theorem \ref{thm:injectivity-result-intro} if one uses first the following result and then applies the isomorphism \eqref{eq:comparison-L-M}.

\begin{theorem}\label{thm:general-totaro-result} For each $p\geq0$, there is a smooth complex projective variety $X$ of dimension $p+3$ such that the group $L^{p+2}H^{p+4}(X)/n$ contains infinitely many elements of order $n$ for all integers $n\geq2$.\end{theorem}

The special cases $p\in\{0,1\}$ that appear in the above theorem have already been explored; see \cite{totaro-annals, Ros08}. Moreover, Theorem \ref{thm:general-totaro-result} also relies on an injectivity result similar to Theorem \ref{thm:injectivity-result-intro}, where we replace the surface $S$, with a very general elliptic curve $E$; see Corollary \ref{cor:ext-prod-elliptic}. Taking products with projective spaces, we obtain a generalization of Theorem \ref{thm:general-totaro-result} in higher dimensions. 

\begin{corollary}\label{cor:general-totaro-result} For each $p\geq0$ and $d\geq p+3,$ there is a smooth complex projective $d$-fold $X$ such that the group $L^{c}H^{2c-p}(X)/n$ contains infinitely many elements of order $n$ for all $n\geq2$ and all $p+2\leq c\leq d-1$.
\end{corollary}

The bounds for $d$ and $c$ given in the above corollary are again optimal for all $p\geq0$. In fact, we have a canonical isomorphism $\CH^{c}(X,p\ ;\Z/n)\cong L^{c}H^{2c-p}(X,\Z/n)$ (see Theorem \ref{thm:suslin-voevodsky} \eqref{it:S-V}) and the former group is known to be finite if either $c\leq p+1$ or $c\geq d$; see Lemma \ref{lem:bound_mod_n} \eqref{it:bound_mod_n}. \par

Besides the case $p=0$, it appears that a variant of Theorem \ref{thm:general-totaro-result} (and of Corollary \ref{cor:general-totaro-result}) for Bloch's higher Chow groups is not known if $p\geq1$. That is, for each $p\geq1,$ a smooth complex projective variety $X$ of dimension $p+3$ such that $\CH^{p+2}(X,p)/\ell$ is infinite for all prime numbers $\ell$. The absence of such a result explains the necessity of working with morphic cohomology in the current work. Furthermore, regarding the case $p=0$ in Theorem \ref{thm:main}, we are unaware of a single example of a smooth complex projective $4$-fold $X$ such that the group $\CH^{3}(X)[\ell]/\ell\CH^{3}(X)[\ell^2]$ is infinite for some prime number $\ell$. This would be consistent with the expected optimal lower bounds. Note that the same question for $\CH^{3}(X)[\ell]$ has an affirmative answer by \cite[Corollary 3.4]{schoen-products}.
 
As in \cite{Ale-griffiths} the surfaces considered in Theorem \ref{thm:injectivity-result-intro} arise as $\Z/n$-quotients of smooth complete intersections inside $\mathbb{P}^{5}_{\C}$. However, the main difference between them is that they admit different degenerations. In \cite[Theorem 6.1]{Sch-griffiths}, to show that the obstruction to the injectivity of the exterior product map vanishes, the author needed a particular kind of semi-stable degenerations of algebraic surfaces. Roughly speaking, that there are certain order $n$ Brauer classes in the geometric generic fibre of the family that extend across the whole family and while the central fibre splits up into several components, the restriction of these lifts to each component is zero. In \cite[Theorem 7.1]{Sch-griffiths}, it was observed that the so-called flower--pot degenerations of the Enriques surfaces \cite[Proposition 3.3.1(3)]{persson} enjoy such properties. Later, this geometric construction was further generalized in \cite[$\S$4]{Ale-griffiths}.\par

In our case, the obstruction is given by the torsion subgroup $\NS(S)_{\tors}$ of the N\'eron--Severi group; see Theorem \ref{thm:injectivity-CH^d(X,1)}. Namely, the surfaces of Theorem \ref{thm:injectivity-result-intro} admit degenerations with the following peculiarities: there exists an \'etale cyclic cover of degree $n$ of $S$ whose associated class in $\NS(S)_{\tors}\cong\Z/n\Z$ is a generator and while the covering map extends as an \'etale cover along the whole family, it becomes trivial over each component of the central fibre; see Theorem \ref{thm:degeneration}. Explicit examples of such surfaces include an Enriques surface that degenerates into a rational polyhedron whose dual graph is a triangulation of $\mathbb{R}\CP^2$, i.e., a type III degeneration; see \cite{kulikov, morrison}. In contrast to \cite{Sch-griffiths}, we work with degenerations that are not necessarily semi-stable; see Remark \ref{remark:semi-stable}. Finally, we also note that the construction from \cite[Theorem 4.1]{Ale-griffiths} could not be used here, as the universal covers of the surfaces in loc. cit. ramify when specialize in the central fibre. 

\par
The paper is organized as follows. After covering the preliminaries in $\S$\ref{sec:preliminaries}, we define actions of higher Chow groups on refined unramified cohomology in $\S$\ref{sec:pairings}. In $\S$\ref{sec:inj}, we provide general injectivity results for such actions. In $\S$\ref{sec:construction}, we introduce the surfaces that we use to prove Theorem \ref{thm:injectivity-result-intro}. Lastly, in $\S$\ref{sec:main-results}, we prove our main results, including Theorems \ref{thm:main}, \ref{thm:injectivity-result-intro} and \ref{thm:general-totaro-result}.

\subsection{Notation} The $n$-torsion subgroup of an abelian group $A$ is denoted by $A[n]:=\{a\in A |\ na=0\}$. If $\ell$ is a prime number, then $A[\ell^{\infty}]:=\{a\in A|\ \ell^{r}a=0\ \text{for some }r\geq1\}$ is the subgroup of $\ell$-primary torsion elements of $A$. By slight abuse of notation, the cokernel of a homomorphism of abelian groups $\psi:A\to B$ will be denoted by $B/A:=\coker(\psi)$.\par
Let $k$ be a field. An algebraic $k$-scheme $X$ is a separated scheme of finite type over $k$ and
it is equi-dimensional if its irreducible components have all the same dimension. A variety over $k$ is an
integral algebraic $k$-scheme. If $k\subset K$ is a field extension and $X$ is a scheme over $k$, then we also set $X_{K}:=X\times_{\Spec k}\Spec K$. 

\section{Preliminaries}\label{sec:preliminaries} 
We gather some of the important properties of refined unramified cohomology that we need for later use. This also includes the specialization maps constructed in \cite{Sch-griffiths}. The reader is advised to consult \cite{Sch-refined, Sch-moving, alexandrou-schreieder, kok-zhou}, from which most of the material we discuss here comes from.
\par Throughout this article, we work with the following cohomology theories.
\subsection{Continuous \'etale cohomology}\label{subsec:cont-etale-cohomology} Let $X$ be a scheme and let $n\geq2$ be an integer invertible in $X$. We write \begin{align} \label{eq:cohomology-mu_n}
    H^{i}(X,\Z/n(j)):= R^{i}\Gamma(X_{\et},\mu^{\otimes j}_{n})
\end{align} 
for the \'etale cohomology groups of the sheaf $\mu^{\otimes j}_{n}$; see \cite{milne}. Here, we recall that $\mu_{n}$ is the \'etale sheaf of $n$-th roots of unity and that its twists $\mu^{\otimes j}_{n}$ are defined, as follows:
$$\mu^{\otimes j}_{n}:=\begin{cases} \mu_{n}\otimes\cdots\otimes\mu_{n}\ (j\text{-times}),\ \text{if}\ j\geq1\\
\Z/n\Z,\ \text{if}\ j=0\\
\Hom(\mu^{\otimes -j}_{n}, \Z/n\Z),\ \text{otherwise.}
\end{cases}$$
\par If $\ell$ is a prime number, then we also set \begin{align}\label{eq:jannsen-cohomology} H^{i}_{\cont}(X,\Z_{\ell}(j)):=\R^{i}(\Gamma\circ\varprojlim)(X_{\et},(\mu_{\ell^{r}}^{\otimes j})_{r}),\end{align} for Jannsen's continuous $\ell$-adic \'etale cohomology; see \cite{jannsen}. By slight abuse of notation, we prefer to write \begin{align}\label{eq:jannsen-cohomology-any-n}H^{i}_{\cont}(X,\Z_{n}(j)):=\bigoplus_{\ell|n}H^{i}_{\cont}(X,\Z_{\ell}(j)),\end{align} where $n\geq 2$ is again a natural number invertible in $X$.

\subsection{Motivic cohomology}\label{subsec:motivic-cohomology} Let $k$ be a field and let $X$ be a smooth and equi-dimensional algebraic $k$-scheme. We denote by $z^{\ast}(X,\bullet),$ the cycle complex of abelian groups introduced by Bloch \cite{bloch-motivic}, whose homology groups define the higher Chow groups $$\CH^{\ast}(X,p):=H_{p}(z^{\ast}(X,\bullet)).$$ Similarly, if $A$ is any abelian group, then $\CH^{\ast}(X,p\ ;A)$ are given by the homology groups of $z^{\ast}(X,\bullet)\otimes A$. The complex $z^{\ast}(X,\bullet)$ is contavariantly functorial for flat morphisms and thus defines a complex of sheaves $z^{\ast}_{X}(\bullet)_{\tau}$ on the small site $X_{\tau},$ where $\tau\in\{\zar,\et\}$.\par We define motivic and Lichtenbaum cohomology of $X$ with values in $A$ by \begin{align}\label{eq:motivic}
    H^{p}_{M}(X,A(c)):=H^{p}(X_{\zar}, A_{X}(c)_{\zar})
\end{align}
and \begin{align}\label{eq:lichtenbaum}
    H^{p}_{L}(X,A(c)):=H^{p}(X_{\et},A_{X}(c)_{\et}),
\end{align}
respectively, where $A_{X}(c):=(z^{c}_{X}(\bullet)_{\zar}\otimes^{\mathbb L} A)[-2c]$ is the weight $c$ motivic complex tensored with $A$. By work of Geisser--Levine \cite{geisser-levine-inventiones,geisser-levine}, we have canonical quasi-isomorphisms
\begin{align} \label{eq:etale-motivic-complex}
(\Z/m)_{X}(c)_{\et}\cong \begin{cases}
\mu_m^{\otimes c}\ \ \ \ &\text{if $m$ is coprime to $\operatorname{char}(k)$;}\\
 W_r\Omega_{X,\log}^c[-c]   \ \ \ \ &\text{if $m=q^r$ and $q=\operatorname{char}(k)>0$,}
\end{cases} 
\end{align}
where $W_r\Omega_{X,\log}^c$ is the logarithmic de Rham Witt sheaf defined in \cite{illusie}.
The first quasi-isomorphism relies on two fundamental results; see \cite[Theorem 1.5]{geisser-levine}. A theorem of Suslin--Voevodsky \cite[Corollary 7.8]{suslin-voevodsky}, which correlates Suslin homology with \'etale cohomology as well as a result of Suslin \cite[Theorem 3.2]{suslin}. The latter gives a comparison between the Suslin homology and motivic cohomology. The second quasi-isomorphism (see \cite[Theorem 8.5]{geisser-levine-inventiones}) follows from a central result of Bloch--Gabber--Kato \cite{bloch-kato}, which relates Milnor's $K$-theory and logarithmic de Rham Witt sheaves of
fields. In addition to that, it depends on a comparison result of Nesterenko-Suslin and Totaro between motivic cohomology and Milnor's $K$-theory; see \cite{nesterenko-suslin, totaro-motivic}.

\subsection{Refined unramified cohomology}\label{subsec:refined-unramified} In the current work we only consider refined unramified cohomology \cite{Sch-refined} for the theories given in $\S$\ref{subsec:cont-etale-cohomology} and $\S$\ref{subsec:motivic-cohomology} and only for smooth and equi-dimensional algebraic $k$-schemes. This is enough for the main applications.\par

For a smooth and equi-dimensional algebraic $k$-scheme $X$, we denote by $F_{\ast}X$ the increasing filtration given by $$F_{0}X\subset F_{1}X\subset F_{2}X\subset \cdots\subset F_{\dim X}X=X,\ \text{where}\ F_{c}X=\{x\in X|\ \codim_{X}(x)\leq c\}.$$  
Above, the codimension of a point $x\in X$ is given by $$\codim_{X}(x):=\dim X-\dim\overline{\{x\}}.$$ Alternatively, the sets $F_{c}X$ can be viewed as pro-schemes made of open subsets $U\subset X$ that contain all codimension $c$ points of $X$. We set \begin{align}\label{eq:cohomology_F_{c}}
    H^{\ast}(F_{c}X,A(j)):=\varinjlim_{F_{c}X\subset U\subset X}H^{\ast}(U,A(j)),
\end{align}
where the cohomology functor is any of the $\eqref{eq:cohomology-mu_n}-\eqref{eq:lichtenbaum}$ and where for the cases \eqref{eq:cohomology-mu_n} and \eqref{eq:jannsen-cohomology-any-n} the abelian group $A$ equals $\Z/n$ and $\Z_{n}:=\bigoplus_{\ell|n}\Z_{\ell}$ with $n\in k^{\ast}$, respectively. Restriction to open subsets yields a natural map $$H^{\ast}(F_{c+1}X,A(j))\longrightarrow H^{\ast}(F_{c}X,A(j))$$ and one defines \begin{align}\label{eq:refined-unramified}
    H^{\ast}_{c,nr}(X,A(j)):=\im(H^{\ast}(F_{c+1}X,A(j))\longrightarrow H^{\ast}(F_{c}X,A(j)))
\end{align}
as the $c$-th refined unramified cohomology of $X$ with values in $A(j)$.\par

By a conjecture of Kok--Zhou \cite[Conjecture 1.4]{kok-zhou} that was recently proven by the first named author and Schreieder \cite{alexandrou-schreieder}, there are canonical isomorphisms \begin{align}\label{eq:bloch-ogus-isomorphism}
    H^{\ast}_{c,nr}(X,A(j))\cong H^{\ast}(X_{\zar},\tau_{\geq\ast-c}\mathcal{K^{\bullet}}),
\end{align}
where $$\mathcal{K}^{\bullet}:=\begin{cases} \R\pi_{\ast}\mu^{\otimes j}_{n},\ \text{if}\ A=\Z/n\ \text{with}\ n\in k^{\ast}\ \text{invertible and}\ H^{\ast}(-,A(j))\ \text{is given by}\ \eqref{eq:cohomology-mu_n},\\
\R\pi_{\ast}(\bigoplus_{\ell|n}\R\varprojlim(\mu^{\otimes j}_{\ell^{r}})_{r}), \ \text{if}\ A=\Z_{n}:=\bigoplus_{\ell|n}\Z_{\ell}\ \text{with}\ n\in k^{\ast}\ \text{invertible and}\  H^{\ast}(-,A(j))\ \text{is given by}\ \eqref{eq:jannsen-cohomology-any-n},\\
\A_{X}(j)_{\zar}, \ \text{if}\ A\ \text{is any abelian group and}\ H^{\ast}(-,A(j))\ \text{is given by}\ \eqref{eq:motivic},\\
\R\pi_{\ast}\A_{X}(j)_{\et},\ \text{if}\ A\ \text{is any abelian group and}\ H^{\ast}(-,A(j))\ \text{is given by}\ \eqref{eq:lichtenbaum}
\end{cases}$$
and where $\pi:X_{\et}\to X_{\zar}$ is the natural map of sites; see \cite[Theorem 1.2]{alexandrou-schreieder}. This generalizes the Bloch--Ogus isomorphism for classical unramified cohomology, i.e., $c=0$; see \cite{BO}.\par
As a consequence of the isomorphism \eqref{eq:bloch-ogus-isomorphism}, one obtains for free the contravariant functoriality of the groups \eqref{eq:refined-unramified}; see \cite[Corollary 1.6]{alexandrou-schreieder}. That is, if $f:X\to Y$ is any map between smooth and equi-dimensional algebraic $k$-schemes, then there is a functorial pullback map \begin{align}\label{eq:pullback} f^{\ast}: H^{i}_{c,nr}(Y,A(j))\longrightarrow H^{i}_{c,nr}(X,A(j)).
\end{align}
Similarly, the refined unramified cohomology groups are covariant functorial for proper maps. Namely, if $f:X\to Y$ is a proper morphism between smooth and equi-dimensional algebraic $k$-schemes, then there is a natural pushforward map \begin{equation}\label{eq:pushforward-F_c}
    f_{\ast}:H^{i}(F_{c}X,A(j))\longrightarrow H^{i+2r}(F_{c+r}Y,A(j+r))
\end{equation}
where $r:=\dim Y-\dim X$; see \cite[Lemma 2.5]{Sch-griffiths}. This respects the filtration $F_{\ast}$ and thus gives a pushforward map
\begin{align}\label{eq:pushforward} f_{\ast}:H^{i}_{c,nr}(X,A(j))\longrightarrow H^{i+2r}_{c+r,nr}(Y,A(j+r))
\end{align}
for refined unramified cohomology.
\subsection{Higher Chow cycles modulo $n$ and refined unramified cohomology}\label{subsec:cycles-refined} Various comparison theorems that correlate some of the refined unramified cohomology groups \eqref{eq:refined-unramified} with algebraic cycle groups are established in \cite[$\S$7]{Sch-refined}. This extends to cycles of arbitrary codimension previous results of Bloch--Ogus \cite{BO}, Colliot-Th\'elène--Voisin \cite{CTV}, Kahn \cite{kahn}, Voisin \cite{Voi-unramified} and
Ma \cite{Ma,Ma2} that use classical unramified cohomology.\par

Let $X$ be a smooth and equi-dimensional algebraic $k$-scheme and let $n\geq2$ be an integer invertible in $k$. It was observed in \cite[Lemma 7.13]{Sch-refined} that if we set $E^{c}_{n}(X):=\ker(\bar{\cl}^{c}_{X}:\CH^{c}(X)/n\to H^{2c}(X,\Z/n(c))),$ where $\bar{\cl}^{c}_{X}$ is the reduction modulo $n$ of the cycle class map $\cl^{c}_{X}:\CH^{c}(X)\to H^{2c}_{\cont}(X,\Z_{n}(c)),$ then there is a natural isomorphism $$E^{c}_{n}(X)\cong \frac{H^{2c-1}_{c-2,nr}(X,\Z/n(c))}{H^{2c-1}(X,\Z/n(c))}.$$
It was later noted that the above isomorphism is part of the following long exact sequence
\begin{align}\label{eq:long-exact-seq_Z/n} 
    \cdots \longrightarrow \CH^{c}(X,2c-p\ ;\Z/n)\overset{\bar{\cl}_{X}^{c,p}}{\longrightarrow} H^{p}_{L}(X,\Z/n(c))\longrightarrow H^{p}_{p-c-1,nr}(X_{\et},(\Z/n)_{X}(c)_{\et})\overset{\theta}{\longrightarrow}\cdots.
\end{align}
The above was first constructed for quasi-projective $X$ and for invertible $n\in k^{\ast}$ in \cite[Theorem 4.16]{kok-zhou}. This also holds for any smooth $X$ as well as for $n=q^{r}$ with $q=\text{char}(k)>0$, 
as a consequence of the isomorphism \eqref{eq:bloch-ogus-isomorphism} established in \cite[Theorem 1.2]{alexandrou-schreieder}; see \cite[Corollary 1.3]{alexandrou-schreieder}.\par

The following integral version of \eqref{eq:long-exact-seq_Z/n} was also obtained in \cite[Corollary 1.4]{alexandrou-schreieder}
\begin{align}\label{eq:long-exact-seq_Z} 
    \cdots \longrightarrow \CH^{c}(X,2c-p)\overset{\cl_{X}^{c,p}}{\longrightarrow} H^{p}_{L}(X,\Z(c))\overset{\rho}{\longrightarrow} H^{p-1}_{p-c-2,nr}(X_{\et},(\Q/\Z)_{X}(c)_{\et}) \longrightarrow\cdots.
\end{align}
The special cases of \eqref{eq:long-exact-seq_Z} that involve classical unramified cohomology in degrees 3 and 4 have previously been noted in \cite[Theorem 1.1]{kahn96} and
\cite[Remarques 2.10 (2)]{kahn}, respectively.\par

Important for us are the following groups of higher Chow cycles:
\begin{align}\label{eq:E^c,p}
    E^{c,p}_{n}(X):=\ker(\bar{\cl}^{c,p}_{X}:\CH^{c}(X,2c-p)/n\longrightarrow H^{p}_{L}(X,\Z/n(c)))
\end{align}
\begin{align}\label{eq:F^c,p}
F^{c,p}_{n}(X):=\ker(\tilde{\cl}^{c,p}_{X}:\CH^{c}(X,2c-p;\Z/n)\longrightarrow H^{p}_{L}(X,\Z/n(c)))
\end{align}
\begin{align}\label{eq:G^c,p}
G^{c,p}(X):=\ker(\cl^{c,p}_{X}:\CH^{c}(X,2c-p)\longrightarrow H^{p}_{L}(X,\Z(c))).
\end{align}

These are in turn related to refined unramified cohomology using the exact sequences \eqref{eq:long-exact-seq_Z/n} and \eqref{eq:long-exact-seq_Z}. More precisely, we have the following.
\begin{lemma}\label{lem:ker_cl} Let $X$ be a smooth and equi-dimensional algebraic scheme over a perfect field $k$. For any integer $n\geq 2$, with notations as in \eqref{eq:E^c,p}--\eqref{eq:G^c,p}, there are canonical isomorphisms 
\begin{align}\label{eq:ker_cl}
    F^{c,p}_{n}(X)\cong \frac{H^{p-1}_{p-c-2,nr}(X_{\et},(\Z/n)_{X}(c)_{\et})}{H^{p-1}_{L}(X,\Z/n(c))} \ \text{and}\ G^{c,p}(X)\cong \frac{H^{p-2}_{p-c-3,nr}(X_{\et},(\Q/\Z)_{X}(c)_{\et})}{H^{p-1}_{L}(X,\Z(c))},
\end{align}
where the refined unramified terms are defined in \eqref{eq:refined-unramified}. Moreover, there is a natural exact sequence 
  \begin{align}\label{eq:exact-sequence-ker_cl}
      0\longrightarrow E^{c,p}_{n}(X)\longrightarrow F^{c,p}_{n}(X) \longrightarrow G^{c,p+1}(X)
  \end{align}
  via which $E^{c,p}_{n}(X)$ identifies with the subgroup  $$\left\{[\alpha]\in \frac{H^{p-1}_{p-c-2,nr}(X,\Z/n(c))}{H^{p-1}(X,\Z/n(c))}\ \Big{|}\ \alpha\in\im(H^{p}_{L}(X,\Z(c))\overset{\rho}{\to} H^{p-1}_{p-c-2,nr}(X,(\Q/\Z)_{X}(c)_{\et})  ) \right\},$$
  where $\rho$ is the composite map $$H^{p}_{L}(X,\Z(c))\longrightarrow H^{p}_{p-c-2,nr}(X_{\et},\Z_{X}(c)_{\et})\cong H^{p-1}_{p-c-2,nr}(X,(\Q/\Z)_{X}(c)_{\et})$$ that appears in \eqref{eq:long-exact-seq_Z}.
\end{lemma}
\begin{proof} The isomorphisms of \eqref{eq:ker_cl} follow immediately from the exact sequences \eqref{eq:long-exact-seq_Z/n} and \eqref{eq:long-exact-seq_Z}, respectively; see \cite[Corollary 1.3, Corollary 1.4]{alexandrou-schreieder}. For \eqref{eq:exact-sequence-ker_cl}, one considers the following diagram, \begin{equation*}
    \begin{tikzcd}
0 \arrow[r] & {H^{p}_{L}(X,\Z(c))/n} \arrow[r]                                 & {H^{p}_{L}(X,\Z/n(c))} \arrow[r]                                      & {H^{p+1}_{L}(X,\Z(c))[n]} \arrow[r]                           & 0 \\
0 \arrow[r] & {\CH^{c}(X,2c-p)/n} \arrow[r] \arrow[u, "{\bar{\cl}^{p,c}_{X}}"] & {\CH^{c}(X,2c-p\ ;\Z/n)} \arrow[r] \arrow[u, "{\tilde{\cl}^{p,c}_{X}}"] & {\CH^{c}(X,2c-p-1)[n]} \arrow[r] \arrow[u, "{\cl^{p,c}_{X}}"] & 0
\end{tikzcd}
\end{equation*}
where we note that the rows are exact. Thus, the snake lemma implies \eqref{eq:exact-sequence-ker_cl} along with the last claim. The description of the map $\rho$ is explained in the proof of \cite[Corollary 1.4]{alexandrou-schreieder}.  
\end{proof}
\subsection{Higher Chow groups modulo $n$}
We collect some results about the structure of the groups $\CH^{c}(X,p)/n$ and $\CH^{c}(X,p)[n]$. The following follows from Rost and Voevodsky's proof of the Beilinson--Lichtenbaum conjecture; see \cite{norm, Voe-milnor,Voevodsky}.

\begin{theorem}\label{thm:Licthenbaum-Beilinson-conjecture} Let $k$ be a perfect field and let $n\geq 2$ be an integer. Let $X$ be a smooth and equi-dimensional algebraic scheme over $k$. Then the following hold true:\begin{enumerate}
    \item \label{it:L-B} We have natural isomorphisms $\CH^{c}(X,p\ ;\Z/n)\cong H^{2c-p}_{L}(X,\Z/n(c))$ for all $c\leq p$ and $$\CH^{p+1}(X,p\ ;\Z/n)\cong N^{1}H^{p+2}_{L}(X,\Z/n(p+1)).$$  
    \item \label{it:L-B'} Similarly, we have natural isomorphisms $\CH^{c}(X,p)_{\tors}\cong H^{2c-p-1}_{L}(X,\Q/\Z(c))/H^{2c-p-1}_{M}(X,\Q(c))$ for all $c\leq p+1$ and
     $$\CH^{p+2}(X,p)_{\tors}\cong N^{1}H^{p+3}_{L}(X,\Q/\Z(p+2))/H^{p+3}_{M}(X,\Q(p+2)),$$
\end{enumerate}
where $N^{1}H^{\ast}_{L}(X,A(c)):=\ker(H^{\ast}_{L}(X,A(c))\to H^{\ast}_{L}(F_{0}X,A(c)))$.
\end{theorem}
\begin{proof} The claims in \eqref{it:L-B} correspond to the Lichtenbaum--Beilinson Conjecture, the rest are immediate consequences. Away from the characteristic, the conjecture follows from \cite[Corollary 1.2]{geisser-levine} and the Bloch--Kato conjecture proven by Voevodsky \cite{Voevodsky}. Moreover, if $n=q^{r}$ with $q=\Char(k)>0$, then it follows from \cite[Theorem 8.5]{geisser-levine-inventiones}. We explain how to obtain \eqref{it:L-B'}. The long exact Bockstein sequence for motivic cohomology gives the exact sequence
$$H^{2c-p-1}_{M}(X,\Q(c))\longrightarrow H^{2c-p-1}_{M}(X,\Q/\Z(c))\longrightarrow \CH^{c}(X,p)_{\tors}\longrightarrow 0.$$
The above middle term identifies with $H^{2c-p-1}_{L}(X,\Q/\Z(c))$ if $c\leq p+1$ and with $N^{1}H^{p+3}_{L}(X,\Q/\Z(p+2))$ if $c=p+2$ by \eqref{it:L-B}. This proves the second claim. 
\end{proof}

 The following is recalled, as we need it for Lemma \ref{lem:bound_mod_n}.
 
\begin{remark}\label{remark:finiteness}  Let $k$ be a field and let $n\geq 2$ be an integer invertible in $k$. Let $G_{k}:=\Gal(k_{s}/k)$ be the absolute Galois group of the field $k$ and denote by $M(G_{k})$ the abelian category of discrete $G_{k}$-modules. Assume that $G_{k}$ has finite cohomology for all finite $n$-torsion $G_{k}$-modules $M$, i.e., the cohomology groups $H^{\ast}(G_{k},M)$ defined as the right derived functors of $M(G_{k})\to \Ab, N\mapsto N^{G_{k}}$ are finite. Then for any smooth variety $X$ over $k$ and for all integers $i,j,$ the \'etale cohomology group $H^{i}(X_{\et},\Z/n(j))$ is finite. In fact, this follows from the finiteness of $H^{i}((X\times_{k} k_{s})_{\et},\Z/n(j))$ and the Hochschild--Serre spectral sequence 
$$E_{2}^{r,s}=H^{r}(G_{k},H^{s}((X\times_{k} k_{s})_{\et},\Z/n(j)))\implies H^{r+s}(X_{\et},\Z/n(j)).$$
Note that the assumption made on $G_{k}$ is satisfied for separably closed, finite fields and local fields; see \cite[Remark (3.5)]{jannsen}.\end{remark}

\begin{lemma}\label{lem:bound_mod_n} Let $k$ be a field and let $n\geq 2$ be an integer invertible in $k$. Assume that the absolute Galois group $G_{k}:=\Gal(k_{s}/k)$ has finite cohomology for all finite $n$-torsion $G_{k}$-modules (see Remark \ref{remark:finiteness}) and put $\epsilon:=\cd(k)$ for the cohomological dimension of the field $k$. Let $X$ be a smooth and equi-dimensional algebraic scheme over $k$ of dimension $d$. Then the following hold: \begin{enumerate} 
\item\label{it:bound_mod_n} The group $\CH^{c}(X,p\ ;\Z/n)$ is finite if $c\geq d+\epsilon$ or $c\leq p+1$. Moreover, we have a natural isomorphism $\cl:\CH^{c}(X,p\ ;\Z/n)\cong H^{2c-p}(X_{\et},\Z/n(c))$ if $c\geq d+\epsilon$ or $c\leq p$.
\item\label{it:bound_torsion} The $n$-torsion group $\CH^{c}(X,p)[n]$ is finite if $c\geq d+\epsilon$ or $c\leq p+2$.\end{enumerate}
\end{lemma} 
\begin{proof} The \'etale cohomology groups $H^{\ast}(X_{\et},\Z/n(c))$ are all finite by the assumption made on $G_{k}$; see Remark \ref{remark:finiteness}. Thus, the claim \eqref{it:bound_mod_n} for $c\leq p+1$ follows from Theorem \ref{thm:Licthenbaum-Beilinson-conjecture} \eqref{it:L-B}. Next, we consider the long exact sequence \eqref{eq:long-exact-seq_Z/n}, which we rewrite as
$$\cdots\longrightarrow H^{2c-p-1}_{c-p-2,nr}(X,\Z/n(c))\overset{\theta}{\longrightarrow} \CH^{c}(X,p\ ;\Z/n)\longrightarrow H^{2c-p}(X,\Z/n(c))\longrightarrow\cdots.$$ Recall from \eqref{eq:bloch-ogus-isomorphism} that $H^{2c-p-\ast}_{c-p-\ast-1,nr}(X,\Z/n(c))\cong H^{2c-p-\ast}(X_{\zar},\tau_{\geq c+1}\R\pi_{\ast}\mu_{n}^{\otimes c})$. Hence, it suffices to show that if $c\geq d+\epsilon$, then the Zariski complex $\tau_{\geq c+1}\R\pi_{\ast}\mu_{n}^{\otimes c}$ is quasi-isomorphic to zero. Indeed, consider the natural injection $\R^{q}\pi_{\ast}\mu_{n}^{\otimes c}\hookrightarrow \bigoplus_{\eta\in X^{(0)}} \iota_{\eta,\ast}H^{q}(k(\eta),\mu_{n}^{\otimes c})$, where $\iota_{\eta,\ast}H^{q}(k(\eta),\mu_{n}^{\otimes c})$ is the constant sheaf supported on the component whose generic point is $\eta$; see \cite[Theorem 4.2]{BO}. Since the group $H^{q}(k(\eta),\mu_{n}^{\otimes c})$ is zero for $q>d+\epsilon$, we in turn obtain that $\R^{q}\pi_{\ast}\mu_{n}^{\otimes c}=0$ for all $q\geq c+1$, as we wanted; see \cite[\href{https://stacks.math.columbia.edu/tag/0F0V}{Tag 0F0V}]{stacks-project}. This completes the proof of \eqref{it:bound_mod_n}.\par
We now prove \eqref{it:bound_torsion} using \eqref{it:bound_mod_n}. Once again, the long exact Bockstein sequence for motivic cohomology gives the exact sequence,
$$H^{2c-p-1}_{M}(X,\Z(c))\longrightarrow H^{2c-p-1}_{M}(X,\Z/n(c))\longrightarrow \CH^{c}(X,p)[n]\longrightarrow 0.$$
We note that the middle term of the above sequence identifies with $H^{2c-p-1}(X_{\et},\Z/n(c))$ if $c\geq d+\epsilon$ or $c\leq p+1$ by \eqref{it:bound_mod_n}. Furthermore, if $c=p+2$, then it is isomorphic to $N^{1}H^{p+3}(X_{\et},\Z/n(p+2))$; see Theorem \ref{thm:Licthenbaum-Beilinson-conjecture} \eqref{it:L-B}. Hence, the group $\CH^{c}(X,p)[n]$ is finite if $c\geq d+\epsilon$ or $c\leq p+2$, since the \'etale cohomology groups $H^{\ast}(X_{\et}, \Z/n(c))$ are finite. This finishes the proof of the lemma.
\end{proof}

\begin{remark} As an application of Lemma \ref{lem:bound_mod_n}, we find that for a separably closed field $k,$ the group $\CH^{c}(X,p)/n$ is finite if $c\geq d$ or $c\leq p+1,$ whereas $\CH^{c}(X,p)[n]$ is finite if $c\geq d$ or $c\leq p+2$.
\end{remark}

\begin{lemma}\label{lem:CH^{d}(X,1)/n} Let $k$ be a separable closed field and let $n\geq 2$ be an integer invertible in $k$. Let $r\geq1$ be an integer. Let $X$ be a smooth, proper and equi-dimensional algebraic scheme over $k$ of dimension $d$. Then the canonical maps \begin{equation*}\begin{split}
    & \frac{\CH^{d}(X,1)[n^r]}{n\CH^{d}(X,1)[n^{r+1}]}\overset{\cl}{\longrightarrow} \frac{H^{2d-1}_{\cont}(X,\Z_{n}(d))[n^r]}{n H^{2d-1}_{\cont}(X,\Z_{n}(d))[n^{r+1}]}\\
    & \CH^{d}(X,1)_{\tors}/n\overset{\cl}{\longrightarrow} H^{2d-1}_{\cont}(X,\Z_{n}(d))_{\tors}/n 
\end{split}
\end{equation*}
are isomorphisms. Moreover, the inclusion $\CH^{d}(X,1)_{\tors}/n\ \hooklongrightarrow \CH^{d}(X,1)/n$ is an equality.
\end{lemma}
\begin{proof} First note that we have a commutative diagram 
\begin{equation*}
    \begin{tikzcd}
{\CH^{d}(X,2\ ;\Z/n^{r})} \arrow[r, "\delta", two heads] \arrow[d, "\cl"'] & {\CH^{d}(X,1)[n^{r}]} \arrow[d, "\cl"] \\
{H^{2d-2}(X_{\et},\Z/n^{r}(d))} \arrow[r, "\delta", two heads]                 & {H^{2d-1}_{\cont}(X,\Z_{n}(d))[n^{r}]},    
\end{tikzcd}
\end{equation*}
where the former vertical map is an isomorphism by Lemma \ref{lem:bound_mod_n} \eqref{it:bound_mod_n}. This in turn implies the surjectivity of the latter, as the Bockstein maps $\delta$ are also onto. Since we know by the same lemma that the cycle map $\cl:\CH^{d}(X,1\ ;\Z/n)\to H^{2d-1}(X_{\et},\Z/n(d))$ is an isomorphism, we find that, in fact, $\cl:\CH^{d}(X,1)/n\to H^{2d-1}_{\cont}(X,\Z_{n}(d))/n$ is injective. Therefore, both observations imply that the maps $$\frac{\CH^{d}(X,1)[n^{r}]}{n\CH^{d}(X,1)[n^{r+1}]}\overset{\cl}{\longrightarrow}\frac{H^{2d-1}_{\cont}(X,\Z_{n}(d))[n^{r}]}{n H^{2d-1}_{\cont}(X,\Z_{n}(d))[n^{r+1}]}$$ and $\cl:\CH^{d}(X,1)_{\tors}/n\to H^{2d-1}_{\cont}(X,\Z_{n}(d))_{\tors}/n$ are isomorphisms, as claimed.\par
Finally, we show that the inclusion $\CH^{d}(X,1)_{\tors}/n\ \hookrightarrow \CH^{d}(X,1)/n$ is an equality. It suffices to show that these finite groups have the same order. To see this, consider the short exact sequence 
$$0\longrightarrow\CH^{d}(X,1)/n\longrightarrow \CH^{d}(X,1\ ;\Z/n)\longrightarrow \CH^{d}(X)[n]\longrightarrow 0$$
and recall that $\CH^{d}(X,1\ ;\Z/n)\cong H^{2d-1}(X_{\et},\Z/n(d))=H^{2d-1}_{\cont}(X,\Z_{n}(d))/n.$ By Roitman's theorem \cite{roitman,bloch-roitman}, the group $\CH^{d}(X)[n]$ identifies with the $n$-torsion subgroup $\Alb(X)[n]\cong(\Z/n)^{b_{1}}$ of the Albanese variety of $X$, where $b_{1}$ is the first Betti number of $X$; see \cite[Remark 5.2.7]{brauer-group}. Hence, the short exact sequence splits, giving so that $$|\CH^{d}(X,1)/n|=|H^{2d-1}_{\cont}(X,\Z_{n}(d))/n|/n^{b_{1}}=|H^{2d-1}_{\cont}(X,\Z_{n}(d))_{\tors}/n|.$$ Note that for the last equality we used that the torsion free part of $H^{2d-1}_{\cont}(X,\Z_{n}(d))$ has rank $b_{1}$ by Poincaré duality; see \cite[Theorem 1.13]{milne-duality}. Now, the first claim clearly implies the second. The proof is complete. 
\end{proof}
\begin{corollary}\label{cor:CH^{d}(X,1)/n} Let $k$ be a separable closed field and let $n\geq 2$ be an integrer invertible in $k$. Let $X$ be a smooth, proper and equi-dimensional algebraic scheme over $k$ of dimension $d$. Then there is a natural isomorphism of finite abelian groups 
\begin{equation}\label{eq:CH^{d}(X,1)/n-vs-Neron-Severi-group}
    \CH^{d}(X,1)/n\cong \Hom(\NS(X)[n],\bigoplus_{\ell|n}\Q_{\ell}/\Z_{\ell}),
\end{equation}
where $\NS(X)$ is the N\'eron--Severi group of $X$.
\end{corollary}
\begin{proof} We can reduce to the case $n=\ell^{r}$ for some prime number $\ell$ and positive integer $r$. It is then enough to construct a natural isomorphism $H^{2d-1}_{\cont}(X,\Z_{\ell}(d))[\ell^{\infty}]/\ell^{r}\cong\Hom(\NS(X)[\ell^{r}],\Q_{\ell}/\Z_{\ell})$; see Lemma \ref{lem:CH^{d}(X,1)/n}. The Poincar\'e duality theorem \cite[Chapter VI, Corollary 11.2]{milne} for $X$ gives a canonical isomorphism $$H^{2d-2}(X_{\et},\Z/\ell^{r}(d))\cong\Hom(H^{2}(X_{\et},\Z/\ell^{r}),\Z/\ell^{r}).$$ This in turn is compatible with respect to the natural inclusion $\mu_{\ell^{r}}^{\otimes d}\hookrightarrow\mu_{\ell^{r+1}}^{\otimes d}$ in the first argument, the natural surjection $\Z/\ell^{r+1}\twoheadrightarrow\Z/\ell^{r}$ in the second argument and the natural inclusion $\Z/\ell^{r}\hookrightarrow\Z/\ell^{r+1}$ in the third. Thus, passing to the limit we obtain a canonical isomorphism \begin{equation}\label{eq:duality}H^{2d-2}(X_{\et},\Q_{\ell}/\Z_{\ell}(d))\cong\Hom(H^{2}_{\cont}(X_{\et},\Z_{\ell}),\Q_{\ell}/\Z_{\ell}).\end{equation} Now, taking the quotient of these groups by their maximal divisible subgroups, \eqref{eq:duality} yields an isomorphism 
\begin{equation}\label{eq:duality-tors}
    H^{2d-1}_{\cont}(X,\Z_{\ell}(d))[\ell^{\infty}]\cong \Hom(\NS(X)[\ell^{\infty}],\Q_{\ell}/\Z_{\ell}),
\end{equation}
where here, we use that $\cl^{1}_{X}:\NS(X)[\ell^{\infty}]\to H^{2}_{\cont}(X,\Z_{\ell})[\ell^{\infty}]$ is an isomorphism; see \cite[(5.13)]{brauer-group}. To prove the claim, it remains to show that $\Hom(\NS(X)[\ell^{\infty}],\Q_{\ell}/\Z_{\ell})\otimes\Z/\ell^{r}\cong \Hom(\NS(X)[\ell^{r}],\Q_{\ell}/\Z_{\ell}).$ The latter follows if one considers the exact sequence $$0\to \NS(X)[\ell^{r}]\to\NS(X)[\ell^{\infty}]\overset{\times\ell^{r}}{\to}\NS(X)[\ell^{\infty}]$$ and then applies the exact functor $\Hom(-,\Q_{\ell}/\Z_{\ell})$. The proof is complete.
\end{proof}

We close this section with some comparison results between morphic cohomology and Bloch's higher Chow groups. We recommend \cite{Fri91, FL92, FM94, suslin-voevodsky, Fri00} for readers interested in morphic cohomology or Lawson homology theory.

\begin{theorem}\label{thm:suslin-voevodsky} Let $X$ be a smooth complex variety. Then for all $c,p\geq0$ and for all $n\geq2$, there are canonical isomorphisms:
\begin{enumerate}
    \item \label{it:S-V} $\CH^{c}(X,p\ ;\Z/n)\cong L^{c}H^{2c-p}(X,\Z/n)$ and
    \item \label{it:S-V'} $\CH^{c}(X,p)[n]/n\CH^{c}(X,p)[n^2]\cong L^{c}H^{2c-p}(X)[n]/nL^{c}H^{2c-p}(X)[n^2].$
\end{enumerate}
\end{theorem}
\begin{proof} The first isomorphism \eqref{it:S-V} is a theorem of Suslin--Voevodsky; see \cite[Theorem 9.1]{suslin-voevodsky}. The latter isomorphism \eqref{it:S-V'} is a consequence of the former. Namely, one considers the following commutative diagram
\begin{equation*}
    \begin{tikzcd}
{\CH^{c}(X,p+1\ ;\Z/n)} \arrow[r, "\delta", two heads] \arrow[d, "\cong"] & {\CH^{c}(X,p)[n]} \arrow[d] \\
{L^{c}H^{2c-p-1}(X,\Z/n)} \arrow[r, "\delta", two heads]                  & {L^{c}H^{2c-p}(X)[n]},      
\end{tikzcd}
\end{equation*}
where the two Bockstein maps $\delta$ are surjective and where the former vertical map is an isomorphism by \eqref{it:S-V}. Thus, the latter vertical map is onto. Again, using the isomorphism in \eqref{it:S-V}, we find that the canonical map $\CH^{c}(X,p)/n\to L^{c}H^{2c-p}(X)/n$ is injective. Hence, all together, they imply the isomorphism in \eqref{it:S-V'}, as we wanted. The proof of Theorem \ref{thm:suslin-voevodsky} is complete.
\end{proof}

\begin{lemma}\label{lem:properties} Let $X$ be a smooth complex projective variety of dimension $d$ such that $\NS(X)_{\tors}\cong\Z/n$ for some $n\geq2$. Then the following hold:\begin{enumerate}
    \item \label{it:Pr-a} We have an equality $\CH^{d}(X,1)[n]/n\CH^{d}(X,1)[n^2]=\CH^{d}(X,1)/n$ and an isomorphism $$\CH^{d}(X,1)/n\cong\Z/n.$$
    \item \label{it:Pr-b} Moreover, if $\Pic(X)=\NS(X)$, then we have an equality $L^{d}H^{2d-1}(X)[n]/nL^{d}H^{2d-1}(X)[n^2]=L^{d}H^{2d-1}(X)/n$ and a canonical isomorphism $$L^{d}H^{2d-1}(X)/n\cong \CH^{d}(X,1)/n\cong\Z/n.$$
\end{enumerate}
\end{lemma}
\begin{proof} We note that \eqref{it:Pr-b} is a consequence of \eqref{it:Pr-a}. Indeed, the Bockstein sequence for morphic cohomology gives a natural isomorphism $L^{d}H^{2d-1}(X)/n\cong L^{d}H^{2d-1}(X,\Z/n)$, since $L^{d}H^{2d}(X)\cong\Z$ is torsion-free. On the other hand, the condition $\Pic(X)=\NS(X)$ implies that the Albanese variety of $X$ is trivial and we in turn find from Roitman's theorem that the torsion in the Chow group of $0$-cycles $\CH_{0}(X)_{\tors}$ is zero as well; see \cite{roitman, bloch-roitman}. Thus, the Bockstein sequence for motivic cohomology yields an isomorphism $\CH^{d}(X,1)/n\cong \CH^{d}(X,1\ ;\Z/n)$ and we see that \eqref{it:Pr-a} implies \eqref{it:Pr-b}; see Theorem \ref{thm:suslin-voevodsky}.\par
Finally, we show \eqref{it:Pr-a}. The condition $\NS(X)_{\tors}\cong\Z/n$ implies that $\CH^{d}(X,1)/n\cong\Z/n$; see Corollary \ref{cor:CH^{d}(X,1)/n}. Recall that the cycle map $\cl^{1}_{X}:\NS(X)\to H^{2}_{\sing}(X,\Z)$ yields an isomorphism on the torsion subgroups. Moreover, we see that the torsion in $H^{2}_{\sing}(X,\Z)$ is the same with the torsion in the singular homology group $H_{1}(X,\Z)$ by the universal coefficient theorem. Thus, Poincar\'e duality in turn gives an isomorphism $H^{2d-1}(X,\Z)_{\tors}\cong\Z/n$. Now, the claim \eqref{it:Pr-a} clearly follows from Lemma \ref{lem:CH^{d}(X,1)/n} together with a comparison theorem between continuous $\ell$-adic \'etale cohomology and singular cohomology; cf. \cite[Chapter III, Theorem 3.12]{milne}. The proof of Lemma \ref{lem:properties} is complete.
\end{proof}
\subsection{Specialization map}\label{subsec:specialization_maps} In this section we recall the construction of the specialization map from \cite[$\S 4$]{Sch-griffiths} and collect some of its properties that we need.\par
Let $k$ be a field, and let $R=\mathcal{O}_{C,x_{0}}$ be the local ring of a smooth $k$-curve $C$ at a closed point $x_{0}\in C$. This is a discrete valuation ring, and we fix a uniformizer $\pi\in R$. We also set $K=\Frac(R)$ for its fraction field.  Let $p:\mathcal{X}\to\Spec R$ be a smooth family whose central fibre $X_{0}:=\mathcal{X}\times_{R}k$ and generic fibre $X_{\eta}:=\mathcal{X}\times_{R}K$ are both equi-dimensional.\par 
We first recall the construction of the specialization map $sp$ for \'etale cohomology with values in $\Z/n$ and $n\in k^{\ast}$.

\subsubsection{Construction of $sp$ for cohomology}
The Gysin exact sequence \cite[$\S 2.3.2$]{brauer-group} for the embedding $X_{0}\hookrightarrow\mathcal{X}$ is given by \begin{equation} \label{eq:gysin}
\cdots \longrightarrow H^{i}(\mathcal{X},\Z/n(j))\longrightarrow H^{i}(X_{\eta},\Z/n(j))\overset{\partial}{\longrightarrow} H^{i-1}(X_{0},\Z/n(j-1))\overset{\iota_{\ast}}{\longrightarrow}\cdots
\end{equation}
and we refer to its co-boundary \begin{equation}\label{eq:residue-map}
\partial:H^{i}(X_{\eta},\Z/n(j))\longrightarrow H^{i-1}(X_{0},\Z/n(j-1)) \end{equation} as the residue map.\par
We then define the specialization map with respect to the choice of the uniformizer $\pi\in R$ by the formula 
\begin{equation}\label{eq:sp-cohomology}
    sp^{\pi}: H^{i}(X_{\eta},\Z/n(j))\longrightarrow H^{i}(X_{0},\Z/n(j)), \alpha\mapsto -\partial(\alpha \cup (\pi)),
\end{equation}
where we view $\pi$ as an element of $H^{1}(K,\Z/n(1))=K^{\ast}/{K^{\ast}}^{n}$ and keep the notation $(\pi):=p^{\ast}\pi\in H^{1}(X_{\eta},\Z/n(1))$ for its pullback; see \cite[$\S$4.1]{Sch-griffiths}.

\begin{remark}\label{rem:sp-alg} If $k$ is algebraically closed, then $sp^{\pi}$ no longer depends on $\pi$ and so we simply write $sp$, instead. In fact, if $\pi'\in R$ is another uniformizer, then $u:=\pi'\pi^{-1}\in R^{\ast}$ is a unit. It follows that $sp^{\pi'}-sp^{\pi}=-(\bar{u})\cup\partial,$ where $\bar{u}\in H^{1}(R/\pi,\Z/n(1))$ denotes the image of $u$ in the residue field $R/\pi$; cf. Lemma \ref{lem:sp_global_class}. Thus, if $k$ is algebraically closed, then $k=R/\pi$ and $H^{1}(k,\Z/n(1))=k^{\ast}/{k^{\ast}}^{n}=0,$ which in turn implies $sp^{\pi'}-sp^{\pi}=-(\bar{u})\cup\partial=0$.
\end{remark}
\par

The following lemma is useful. 
\begin{lemma}\label{lem:sp_global_class} With notations as above, if $\alpha\in H^{p}(\mathcal{X},\Z/n(m))$, then for every $\beta\in H^{i}(X_{\eta},\Z/n(j)),$ we have that $$sp^{\pi}(\alpha|_{X_{\eta}}\cup\beta)=\alpha|_{X_{0}}\cup sp^{\pi}(\beta).$$ In particular, $sp^{\pi}(\alpha|_{X_{\eta}})=\alpha|_{X_{0}}$.
\end{lemma}
\begin{proof} The last claim follows from the first if one puts $\beta=1\in H^{0}(X_{\eta},\Z/n)=\Z/n$ and note that $\partial(\pi)=-1\in H^{0}(X_{0},\Z/n)=\Z/n $; see \cite[Lemma 4.4]{Sch-griffiths}. The first claim follows from \cite[Lemma 2.4]{Sch-unramified}.\end{proof}

\begin{remark}The definition \eqref{eq:sp-cohomology} should be compared with Rost's definition of specialization maps for cycle modules; see \cite[pp. 328]{rost}.
\end{remark}
\subsubsection{Construction of $sp$ for refined unramified cohomology} The analogous construction can be carried out for the \'etale cohomology of the filtration $F_{\ast}$  \eqref{eq:cohomology_F_{c}} with values in $A=\Z/n$. First, since the cup product is compatible with the restriction to open subsets, we obtain well defined maps 
\begin{equation}\label{eq:cup_product_pi-F_c}
    -\cup (\pi): H^{i}(F_{c}X_{\eta},\Z/n(j))\longrightarrow H^{i+1}(F_{c}X_{\eta},\Z/n(j+1)),
\end{equation}
where we recall that the filtration $F_{\ast}$ is defined in $\S$\ref{subsec:refined-unramified}. 
The usual Gysin exact sequence \cite[$\S 2.3.2$]{brauer-group} also induces a Gysin sequence for the cohomology of the filtration $F_{\ast}$ \eqref{eq:cohomology_F_{c}}; see \cite[Proposition 3.2 (11) $\&$ Remark 3.3]{kok-zhou}. That is, we have an exact sequence
\begin{equation}\label{eq:gysin-F_{c}}
    \cdots \longrightarrow H^{i}(F_{c}\mathcal{X},\Z/n(j))\longrightarrow H^{i}(F_{c}X_{\eta},\Z/n(j))\overset{\partial}{\longrightarrow} H^{i-1}(F_{c}X_{0},\Z/n(j-1))\overset{\iota_{\ast}}{\longrightarrow}\cdots.
\end{equation}

As in \eqref{eq:sp-cohomology}, we then define specialization maps with respect to the choice of the uniformizer $\pi\in R$ by the formula
\begin{equation}\label{eq:sp-F_c}
sp^{\pi}: H^{i}(F_{c}X_{\eta},\Z/n(j))\longrightarrow H^{i}(F_{c}X_{0},\Z/n(j)),\ \alpha\mapsto -\partial(\alpha\cup (\pi)),
\end{equation}
where the cup product $\cup (\pi)$ is given by \eqref{eq:cup_product_pi-F_c} and where the residue map $\partial$ is simply the co-boundary map of \eqref{eq:gysin-F_{c}}.

We have the following list of properties. 
\begin{proposition}{$($\cite[Proposition 4.1]{Sch-griffiths}$)$}\label{prop:specialization} With notations as fixed at the beginning of $\S$\ref{subsec:specialization_maps}, the specialization map 
$$sp^{\pi}: H^{i}(F_{c}X_{\eta},\Z/n(j))\longrightarrow H^{i}(F_{c}X_{0},\Z/n(j))$$
from \eqref{eq:sp-F_c} satisfies the following properties:\begin{enumerate}
    \item\label{item:sp-k-alg} $sp^{\pi}$ does not depend on the uniformizer $\pi$ if $k$ is algebraically closed,
    \item\label{item:sp-F_*} $sp^{\pi}$ respects the filtration $F_{\ast}$,
    \item\label{item:sp-global-class} For any class $[\alpha]\in H^{p}_{c,nr}(\mathcal{X},\Z/n(m))$ and any class $\beta\in H^{i}(F_{c}X_{\eta},\Z/n(j)),$ we have that $sp^{\pi}(\alpha|_{F_{c}X_{\eta}}\cup\beta)=\alpha|_{F_{c}X_{0}}\cup sp^{\pi}(\beta)$. In particular, $sp^{\pi}(\alpha|_{F_{c}X_{\eta}})=\alpha|_{F_{c}X_{0}}$.
    \item\label{item:sp-pushforward} The specialization map $sp^{\pi}$ commutes with the pushforward maps given in \eqref{eq:pushforward-F_c}, i.e., if $\mathcal{Y}\to\Spec R$ is another smooth family with equi-dimensional fibres and $f:\mathcal{X}\to\mathcal{Y}$ is a proper $R$-morphism of pure relative dimension, then $(f|_{X_{0}})_{\ast}\circ sp^{\pi}=sp^{\pi}\circ (f|_{X_{\eta}})_{\ast}$.
    \item\label{item:sp-geometric} The map $sp^{\pi}$ induces a specialization map between the geometric fibres $X_{\bar{\eta}}=X_{\eta}\times\bar{K}$ and $X_{\bar{0}}=X_{0}\times_{k}\bar{k}$:$$\bar{sp}: H^{i}(F_{c}X_{\bar{\eta}},\Z/n(j))\longrightarrow H^{i}(F_{c}X_{\bar{0}},\Z/n(j)).$$
    In addition, this respects the filtration $F_{\ast}$ and does not depend on the choice of uniformizer $\pi\in R$.
\end{enumerate}
\end{proposition}
\begin{proof} For the proof see \cite[Proposition 4.1]{Sch-griffiths}. We only explain \eqref{item:sp-global-class}. Since $\alpha\in H^{p}(F_{c+1}\mathcal{X},\Z/n(m)),$ we may find an open subset $\mathcal{U}\subset\mathcal{X}$ such that $F_{c+1}\mathcal{X}\subset \mathcal{U}$ and such that $\alpha\in  H^{p}(\mathcal{U},\Z/n(m))$. In particular, we have that $F_{c}X_{0}\subset \mathcal{U}$. Now let $\beta\in H^{i}(F_{c}X_{\eta},\Z/n(j))$ and pick a closed subset $W\subset X_{\eta}$ such that $\dim X_{\eta}-\dim W\geq c+1$ and such that $\beta\in H^{i}(U_{\eta}\setminus W,\Z/n(j))$, where $U_{\eta}=\mathcal{U}\times_{k}K$ (and here note that we explicitly used that $F_{c}U_{\eta}=F_{c}X_{\eta}$).  We set $\mathcal{W}\subset\mathcal{U}$ for the closure of $W$ inside $\mathcal{U}$ and denote by $W_{0}:=\mathcal{W}\times_{R}k$ and $W_{\eta}=\mathcal{W}\times_{R}K$ its central and generic fibre, respectively. We then simply apply Lemma \ref{lem:sp_global_class} to the smooth family $\mathcal{U}\setminus\mathcal{W}\to\Spec R$ and classes $\alpha|_{\mathcal{U}\setminus\mathcal{W}}\in H^{p}(\mathcal{U}\setminus\mathcal{W},\Z/n(m))$ and $\beta\in H^{i}(U_{\eta}\setminus W_{\eta},\Z/n(j))$. That is, we get $sp^{\pi}(\alpha|_{U_{\eta}\setminus W_{\eta}}\cup\beta)=\alpha|_{U_{0}\setminus W_{0}}\cup sp^{\pi}(\beta)$. This in turn implies what we want as $F_{c}X_{\eta}\subset U_{\eta}\setminus W_{\eta}$ and $F_{c}X_{0}\subset U_{0}\setminus W_{0}$. The second claim of \eqref{item:sp-global-class} follows from the first; cf. proof of Lemma \ref{lem:sp_global_class}.
\end{proof}
\begin{remark}\label{rem:sp-refined} By \eqref{item:sp-F_*} and \eqref{item:sp-geometric} one obtains well-defined specialization maps $$sp^{\pi}: H^{i}_{c,nr}(X_{\eta},\Z/n(j))\longrightarrow H^{i}_{c,nr}(X_{0},\Z/n(j))\ \text{and}\ \bar{sp}: H^{i}_{c,nr}(X_{\bar{\eta}},\Z/n(j))\longrightarrow H^{i}_{c,nr}(X_{\bar{0}},\Z/n(j))$$
for refined unramified cohomology; see \cite[Corollary 4.2]{Sch-griffiths}. The properties \eqref{item:sp-k-alg}--\eqref{item:sp-geometric} can also be formulated accordingly for these maps. In particular, $\bar{sp}$ does not depend on the choice of uniformizer $\pi\in R$.
\end{remark}

\section{Pairings}\label{sec:pairings}   
In this section inspired by \cite[$\S4$]{Sch-moving}, we construct actions of higher Chow cycles on refined unramified cohomology. 
We are interested to study the injectivity property of the following pairing.

\begin{lemma}\label{lem:pairing} Let $k$ be a perfect field and let $H^{\ast}(-,A(j))$ be any of the cohomology functors \eqref{eq:cohomology-mu_n}--\eqref{eq:lichtenbaum}. Then for any pair of smooth and equi-dimensional algebraic $k$-schemes $X$ and $Y$ and integers $c,p,i,r\geq0,$ there is a natural pairing
\begin{equation}\label{eq:pairing}
    \CH^{c}(X,p\ ; A)\otimes H^{i}_{r,nr}(Y,A(j))\overset{\times}{\longrightarrow} H^{i+2c-p}_{r+c-p,nr}(X\times Y,A(c+j))
\end{equation}
    which is compatible with the pairing
    \begin{equation*}
        \CH^{c}(X,p\ ; A)\otimes H^{i}(Y,A(j))\longrightarrow H^{i+2c-p}(X\times Y,A(c+j)),\ [z]\otimes \alpha \mapsto \cl^{2c-p,c}_{X\times Y}(p^{\ast}[z])\cup q^{\ast}\alpha,
    \end{equation*}
    where $p:X\times_{k} Y\to X$ and $q:X\times_{k}Y\to Y$ are the canonical projection maps. That is, the diagram below
    \begin{equation*}
        \begin{tikzcd}
{\CH^{c}(X,p\ ;A)\otimes H^{i}(Y,A(j))} \arrow[r] \arrow[d] & {H^{i+2c-p}(X\times Y,A(j+c))} \arrow[d]  \\
{\CH^{c}(X,p\ ;A)\otimes H^{i}_{r,nr}(Y,A(j))} \arrow[r]    & {H^{i+2c-p}_{r+c-p,nr}(X\times Y,A(j+c))}
\end{tikzcd}
    \end{equation*}
    commutes.
\end{lemma}
\begin{proof} Let $\mathcal{Z}^{c}(X,\bullet)\otimes A$ be Bloch's cycle complex of abelian groups whose homology groups compute higher Chow groups; see $\S$\ref{subsec:motivic-cohomology}. Let $d_{\ast}:\mathcal{Z}^{c}(X,\ast)\otimes A\to \mathcal{Z}^{c}(X,\ast-1)\otimes A$ be its differential. We distinguish the cases $c>p$ and $c\leq p$. We first treat the case $c\leq p$. Then the compatibility of the cup product with respect to restriction on an open subset yields a well defined exterior product map 
$$\Lambda : H^{2c-p}(X,A(c))\otimes H^{i}_{r,nr}(Y,A(j))\longrightarrow H^{2c-p+i}_{r,nr}(X\times Y,A(c+j)), \alpha\otimes [\beta]\mapsto [p^{\ast}\alpha\cup q^{\ast}\beta].$$
Since we assumed $c-p\leq 0$, we find a natural restriction map $$\res :H^{2c-p+i}_{r,nr}(X\times Y,A(c+j))\longrightarrow H^{2c-p+i}_{r+c-p,nr}(X\times Y,A(c+j))$$ and thus, the pairing \eqref{eq:pairing} can be then defined as the composite map $\res\circ\Lambda\circ(\cl_{X}\otimes\id),$ where $\cl_{X}:\CH^{c}(X,p\ ; A)\to H^{2c-p}(X,A(c))$ is the corresponding cycle class map. 

\par Next, we consider the case $c>p$ and construct a natural pairing
\begin{equation}\label{eq:pairing-I}
    \ker(d_{p}) \otimes H^{i}(F_{r}Y,A(j)) \longrightarrow H^{i+2c-p}(F_{r+c-p}(X\times Y),A(j+c)).
\end{equation}
Pick a cycle $\Gamma\in\ker(d_{p})$ and set $|\Gamma|:=\supp(\Gamma)\subset X\times\Delta^{p}$ for its support. We consider the projection map $\pr_{1}:X\times\Delta^{p}\to X$ and note that $Z:=\overline{\pr_{1}(|\Gamma|)}\subset X$ has codimension at least $c-p$ in $X$. Up to enlarging this closed subscheme, we may assume that $Z\subset X$ is equi-dimensional of codimension-$(c-p)$ in $X$ and that $[\Gamma]\in\CH^{p}(Z,p\ ;A)$. There is a canonical homomorphism $\rho:H^{2c-p}_{Z}(X_{\zar},A_{X}(c)_{\zar})\to H^{2c-p}_{Z}(X,A(c)),$ where $H^{\ast}_{Z}(X,A(c))$ is the appropriate cohomology group with support; see \eqref{eq:etale-motivic-complex}. We thus obtain a cycle class map \begin{equation}\label{eq:cl_Z}
    \cl_{Z}: \CH^{p}(Z,p\ ;\A)\cong H^{2c-p}_{Z}(X_{\zar},A_{X}(c)_{\zar})\overset{\rho}{\longrightarrow} H^{2c-p}_{Z}(X,A(c)) 
\end{equation}
and consider the class $\cl_{Z}(\Gamma)\in H^{2c-p}_{Z}(X,A(c))$.\par
On the other hand, let $\alpha\in H^{i}(F_{r}Y,A(j))$ and choose a lift $\tilde{\alpha}\in H^{i}(U, A(j))$ over some open subset $U\subset X$ whose complement $R:=Y\setminus U$ has codimension at least $r+1$ in $Y$. We consider the following composite 
\begin{equation}\label{eq:cl_{Z}xa}
\begin{split}&H^{i}(U,A(j))\xrightarrow{\text{Exc}\circ\cl_{Z\times U}(p^{\ast}\Gamma)\cup q^{\ast}}H^{i+2c-p}_{Z\times U}(X\times Y\setminus (Z\times R),A(c+j))\\
& \overset{\iota_{\ast}}{\longrightarrow} H^{i+2c-p}(X\times Y\setminus (Z\times R),A(c+j)),
\end{split} 
\end{equation}
where $$\cup: H^{2c-p}_{Z\times U}(X\times U,A(c))\otimes H^{i}(X\times U,A(j))\longrightarrow H^{i+2c-p}_{Z\times U}(X\times U,A(c+j))$$ is the natural cup product pairing for $(H^{\ast},H^{\ast}_{Z})$ and where for the middle term of \eqref{eq:cl_{Z}xa} we use the excision isomorphism $\text{Exc}$
$$\ H^{i+2c-p}_{Z\times U}(X\times Y\setminus (Z\times R),A(c+j))\cong H^{i+2c-p}_{Z\times U}(X\times U,A(c+j)).$$
Now the closed subset $Z\times R\subset X\times Y$ has codimension at least $c-p+r+1$ in $X\times Y$ and thus we can define $\Gamma\times \alpha:=[\iota_{\ast}\text{Exc}(\cl_{Z\times U}(p^{\ast}\Gamma)\cup\tilde{\alpha})]\in H^{i+2c-p}(F_{c-p+r}(X\times Y),A(c+j))$. Note that the class $\Gamma\times\alpha$ is well defined, i.e., it is independent of our choices of a lift $\tilde{\alpha}$ and a closed subscheme $Z\subset X$ such that $|\Gamma|\subset Z\times\Delta^{p}$. In fact, the former follows as any two lifts of $\alpha$ coincide over an open subset that contains $F_{r}Y$. In addition, the covariant functoriality of the cycle map $\cl_{Z}$ for proper maps implies the latter, since a diagram of the form
\begin{equation*}
    \begin{tikzcd}
{H^{i_{1}}_{Z}(X,A(j_{1}))} \arrow[d, "\cup \beta"] \arrow[r, "\iota_{\ast}"] & {H^{i_{1}}_{Z'}(X,A(j_{1}))} \arrow[d, "\cup \beta"] \\
{H^{i_{1}+i_{2}}_{Z}(X,A(j_{1}+j_{2}))} \arrow[r, "\iota_{\ast}"]             & {H^{i_{1}+i_{2}}_{Z'}(X,A(j_{1}+j_{2}))}            
\end{tikzcd}
\end{equation*}
always commutes for $\beta\in H^{i_{2}}(X,A(j_{2}))$ and $Z\subset Z'$ closed subschemes of $X$ (one can see this using the six functor formalism in the resp. derived category for instance).\par
By construction, the diagram 
\begin{equation*}
    \begin{tikzcd}
{ \ker(d_{p})\otimes H^{i}(F_{r+1}Y,A(j))} \arrow[d] \arrow[r, "\times"] & {H^{2c-p+i}(F_{c-p+r+1}(X\times Y),A(c+j))} \arrow[d] \\
{ \ker(d_{p})\otimes H^{i}(F_{r}Y,A(j))} \arrow[r, "\times"]             & {H^{2c-p+i}(F_{c-p+r}(X\times Y),A(c+j))}            
\end{tikzcd}
\end{equation*}
commutes, thus inducing a pairing 
\begin{equation}\label{eq:pairing-II}
    \ker(d_{p}) \otimes H^{i}_{r,nr}(Y,A(j)) \longrightarrow H^{i+2c-p}_{r+c-p,nr}(X\times Y,A(j+c)).
\end{equation}
\par Finally, we show that $\Gamma\times\alpha=0,$ if $\Gamma\in\im(d_{p+1})$. In this case, there is an equi-dimensional closed subscheme $W\subset X$ of codimension-$(c-p-1)$, such that $|\Gamma|\subset W\times\Delta^{p}$ and $\Gamma=0\in\CH^{p+1}(W,p)$. In particular, this implies that its cycle class $\cl_{W}(\Gamma)=0\in H^{2c-p}_{W}(X,A(c))$ vanishes. Let $\tilde{\alpha}\in H^{i}(V,A(c))$ be a lift of the class $\alpha\in H^{i}_{r,nr}(Y,A(j))$ over some open subset $V\subset Y$ whose complement $D:=Y\setminus V$ has codimension at least $r+2$ in $Y$. It can be readily checked that a representative for the class $\Gamma\times\alpha\in H^{i+2c-p}_{r+c-p,nr}(X\times Y,A(j+c))$ is given then as the value of $\tilde{\alpha}$ by the following map
\begin{equation*}
    \begin{split}&H^{i}(V,A(j))\xrightarrow{\text{Exc}\circ\cl_{Z\times U}(p^{\ast}\Gamma)\cup q^{\ast}}H^{i+2c-p}_{W\times V}(X\times Y\setminus (W\times D),A(c+j))\\
& \overset{\iota_{\ast}}{\longrightarrow} H^{i+2c-p}(X\times Y\setminus (W\times D),A(c+j)),
\end{split} 
\end{equation*}
where as before $\text{Exc}$ is the excision isomorphism $$H^{i+2c-p}_{W\times V}(X\times Y\setminus (W\times D),A(c+j))\cong H^{i+2c-p}_{W\times V}(X\times V,A(c+j)).$$ We thus conclude $\Gamma\times\alpha=0,$ since $\cl_{W\times V}(p^{\ast}\Gamma)=p^{\ast}\cl_{W}(\Gamma)=0$. The compatibility property for the pairing $\times$ is clear by construction. The proof is complete.
\end{proof}

\begin{remark}\label{rem:algebraic equivalence} Let $c>p$. In the notation of Lemma \ref{lem:pairing}, if one sets $$N^{c-p-1}\CH^{c}(X,p\ ;A):=\im (\bigoplus_{W\subset X}\ker(\cl_{W})\overset{\iota_{\ast}}{\longrightarrow}\CH^{c}(X,p\ ;A)),$$ where $W\subset X$ runs through equi-dimensional closed subschemes of codimension-$(c-p-1)$ in $X$ and where the map $\cl_{W}:\CH^{p+1}(W,p\ ;A)\cong H^{2c-p}_{W}(X_{\zar},A_{X}(c)_{\zar})\rightarrow H^{2c-p}_{W}(X,A(c))$ is the natural one, then the pairing 
$$N^{c-p-1}\CH^{c}(X,p\ ;A)\otimes H^{i}_{r,nr}(Y,A(j))\longrightarrow H^{i+2c-p}_{r+c-p,nr}(X\times Y, A(c+j))$$ becomes the zero map. This is implicitly implied by the argument in the last paragraph of the proof in the above lemma.\par
It also says that in the special case one considers an algebraically closed field $k,$ takes $A=\Z_{\ell}$ with $\ell\neq\Char(k)$, sets $p=0$ and let $H^{\ast}(-,A(j))$ be the cohomology functor given by $\ell$-adic continuous \'etale cohomology \eqref{eq:jannsen-cohomology}, then \eqref{eq:pairing} yields a well defined pairing 
$$\CH^{c}(X)/\sim_{\alg}\otimes H^{i}_{r,nr}(Y,\Z_{\ell}(j))\longrightarrow H^{i+2c}_{r+c,nr}(X\times Y,\Z_{\ell}(c+j)),$$ where $\CH^{c}(X)/\sim_{\alg}$ is the Chow group modulo algebraic equivalence; see \cite[Lemma 7.5, Proposition 6.6]{Sch-refined}.\end{remark}

Next, we compare the natural exterior product maps for higher Chow groups 
\begin{equation}\label{eq:exterior-product-chow}
    \CH^{c_{1}}(X,p_{1}\ ; A)\otimes \CH^{c_{2}}(Y,p_{2}\ ;A)\longrightarrow \CH^{c_{1}+c_{2}}(X\times Y,p_{1}+p_{2}\ ; A),\ [\Gamma_{1}]\otimes[\Gamma_{2}]\mapsto [\Gamma_{1}\times\Gamma_{2}] 
\end{equation}
with the pairings constructed in Lemma \ref{lem:pairing}. 

\begin{lemma}\label{lem:pairings-comparison} Let $X$ and $Y$ be smooth and equi-dimensional algebraic schemes over a perfect field $k$. Let $n\geq 2$ be an integer. Then the following diagram 
\begin{equation}\label{eq:exterior-products-compatibility}
    \begin{tikzcd}
{\CH^{c}(X,p\ ; \Z/n)\otimes \CH^{j}(Y,2j-i\ ; \Z/n)} \arrow[r, "\eqref{eq:exterior-product-chow}"]                                                & {\CH^{j+c}(X\times Y,p+2j-i\ ; \Z/n)}                                                      \\
{\CH^{c}(X,p\ ;\Z/n)\otimes H^{i-1}_{i-j-2,nr}(Y,(\Z/n)_{Y}(j)_{\et})} \arrow[r, "\eqref{eq:pairing}"] \arrow[u, "\id\otimes\theta"] & {H^{2c-p+i-1}_{c-p+i-j-2,nr}(X\times Y,(\Z/n)_{X\times Y}(c+j)_{\et})} \arrow[u, "\theta"]
\end{tikzcd}
\end{equation}
commutes, where we use the symbol $\theta$ for the coboundary maps of \eqref{eq:long-exact-seq_Z/n}.
\end{lemma}
\begin{proof} We first recall the description of the coboundary map $$\theta: H^{i-1}_{i-j-2,nr}(Y,(\Z/n)_{X}(j)_{\et})\to \CH^{j}(Y,2j-i\ ; \Z/n)$$ from \cite[Proposition 4.15, (27)]{kok-zhou}. Note that although the construction of $\theta$ in \cite[Proposition 4.15]{kok-zhou} is stated only for $n\in k^{\ast}$ invertible, it can also be generalised to $n=q^{r}$ with $q=\Char(k)>0$; cf. \cite[Corollary 1.3]{alexandrou-schreieder}. Let $[\alpha]\in H^{i-1}_{i-j-2,nr}(Y,(\Z/n)_{Y}(j)_{\et})$, where we can assume $\alpha\in H^{i-1}(Y\setminus R,(\Z/n)_{X}(j)_{\et})$ for some equi-dimensional closed subscheme $R\subset Y$ of codimension $\geq(i-j)$. The cycle class map $\cl_{R}: H^{i}_{R}(Y_{\zar}, (\Z/n)_{Y}(j)_{\zar})\to H^{i}_{R}(Y_{\et}, (\Z/n)_{Y}(j)_{\et})$ can be then checked to be an isomorphism. Indeed, this is a formal consequence of the Beilinson--Lichtenbaum conjecture (see Theorem \ref{thm:Licthenbaum-Beilinson-conjecture}), if one considers the relevant Zariski hypercohomology spectral sequences and note that $E^{p,q}_{2}=H^{p}_{R}(Y_{\zar},R^{q}\pi_{\ast}(\Z/n)_{Y}(j)_{\et})=0 $ for $p<i-j$ by a simple argument using the Gersten resolution; see \cite{BO,gros-suwa}. The class $\theta([\alpha])$ is then given as the value of $\alpha$ by the following map
\begin{equation}\label{eq:coboundary-theta}\begin{split}
    &H^{i-1}(Y\setminus R,(\Z/n)_{X}(j)_{\et})\overset{\partial}{\longrightarrow}H^{i}_{R}(Y_{\et}, (\Z/n)_{Y}(j)_{\et})\overset{\cl_{R}^{-1}}{\longrightarrow} H^{i}_{R}(Y_{\zar}, (\Z/n)_{Y}(j)_{\zar})\\
    &\overset{\iota_{\ast}}{\longrightarrow}H^{i}(Y_{\zar}, (\Z/n)_{Y}(j)_{\zar})=\CH^{j}(Y,2j-i\ ; \Z/n),
    \end{split}
\end{equation}
where $\partial$ is the coboundary map of the long exact \'etale cohomology sequence associated to the pair $(R,Y)$; see \cite[III.1.25]{milne}. In particular, we have $\theta(\alpha)=[\iota_{\ast}\cl^{-1}_{R}\partial(\alpha)]$.\par
Next, assume that $c>p$ and let $[\Gamma]\in\CH^{c}(X,p\ ;\Z/n)$. Choose an equi-dimensional closed subscheme $Z\subset X$ of codimension-$(c-p)$ in $X$, such that $\supp(\Gamma)\subset Z\times \Delta^{p}$. Since $\theta$ is compatible with flat pull-backs, we find $$\eqref{eq:exterior-product-chow}\circ (\id\otimes\theta)([\Gamma]\otimes[\alpha])=[p^{\ast}\Gamma\cup \iota_{\ast}\cl^{-1}_{X\times R}\partial(q^{\ast}\alpha)]=\iota_{\ast}[p^{\ast}\Gamma\cup \cl^{-1}_{X\times R}\partial(q^{\ast}\alpha)],$$   
where in the last equality, we view $\Gamma$ as a class in $H^{2c-p}_{Z}(X_{\zar},(\Z/n)_{X}(c)_{\zar})$.\par
On the other hand, the constructive proof of Lemma \ref{lem:pairing}, shows that $$\eqref{eq:pairing}([\Gamma]\otimes[\alpha])=[\iota_{\ast}\text{Exc}(p^{\ast}\cl_{Z}(\Gamma)\cup q^{\ast}\alpha)],$$
where the notation here is consistent with the one used in \eqref{eq:cl_{Z}xa}. We thus obtain 
$$\theta\circ\eqref{eq:pairing}([\Gamma]\otimes[\alpha])= [\iota_{\ast}\cl^{-1}_{Z\times R}\partial(\text{Exc}(p^{\ast}\cl_{Z}(\Gamma)\cup q^{\ast}\alpha))].$$
We finally need to show the equality $\cl^{-1}_{Z\times R}\partial(\text{Exc}(p^{\ast}\cl_{Z}(\Gamma)\cup q^{\ast}\alpha))=p^{\ast}\Gamma\cup \cl^{-1}_{X\times R}\partial(q^{\ast}\alpha)$ or equivalently $\partial(\text{Exc}(p^{\ast}\cl_{Z}(\Gamma)\cup q^{\ast}\alpha))=p^{\ast}\cl_{Z}(\Gamma)\cup \partial(q^{\ast}\alpha)$. The latter translates into a compatibility problem of cup products with residue maps, i.e., it suffices to know if the diagram
\begin{equation*}
   \begin{tikzcd}
{H^{i-1}(X\times (Y\setminus R),(\Z/n)_{X\times Y}(j)_{\et})} \arrow[d, "\text{Exc}\circ(\cl_{Z\times Y}(p^{\ast}\Gamma)\cup -)"] \arrow[r, "\partial"] & {H^{i}_{X\times R}(X\times Y,(\Z/n)_{X\times Y}(j)_{\et})} \arrow[d, "\cl_{Z\times Y}(p^{\ast}\Gamma)\cup -"] \\
{H^{i-1+2c-p}_{Z\times R}(X\times Y\setminus(Z\times R),(\Z/n)_{X\times Y}(j+c)_{\et})} \arrow[r, "\partial"]                                           & {H^{i+2c-p}_{Z\times R}(X\times Y,(\Z/n)_{X\times Y}(j+c)_{\et})}                                           
\end{tikzcd}
\end{equation*}
commutes. This in turn is well known; see \cite[Lemma B.2]{Sch-moving}. This proves the commutativity of \eqref{eq:exterior-products-compatibility} in the case $c>p$.\par

The case $c\leq p$ remains to be considered. In the above discussion, we can then take $Z=X$ and observe that $$\theta\circ\eqref{eq:pairing}([\Gamma]\otimes[\alpha])= [\iota_{\ast}\cl^{-1}_{X\times R}\partial(p^{\ast}\cl_{X}(\Gamma)\cup q^{\ast}\alpha)].$$ Thus, as before it is enough to note that $\partial(p^{\ast}\cl_{X}(\Gamma)\cup q^{\ast}\alpha)=p^{\ast}\cl_{X}(\Gamma)\cup \partial(q^{\ast}\alpha)$. The latter equality follows from \cite[Lemma 2.4]{Sch-unramified}. The proof of the lemma is now complete.
\end{proof}

\begin{corollary}\label{cor:pairings-comparison} Let $X$ and $Y$ be smooth and equi-dimensional algebraic schemes over a perfect field $k$. Let $n\geq2$ be an integer and consider the natural exterior product map \begin{equation}\label{eq:exterior-product-Fn}
    \CH^{c}(X,p \ ;\Z/n)\otimes F^{j,i}_{n}(Y)\longrightarrow F^{c+j,2c-p+i}_{n}(X\times Y),\ [\Gamma_{1}]\otimes[\Gamma_{2}]\mapsto [\Gamma_{1}\times \Gamma_{2}]
\end{equation}
induced from \eqref{eq:exterior-product-chow}, where recall that the cycle group $F^{j,i}_{n}(Y)$ is defined in \eqref{eq:F^c,p}. Then \eqref{eq:exterior-product-Fn} identifies through the isomorphisms \eqref{eq:ker_cl} with the pairing \begin{equation}\label{eq:pairing-refined}
\CH^{c}(X,p \ ;\Z/n)\otimes\frac{ H^{i-1}_{i-j-2,nr}(Y,(\Z/n)_{Y}(j)_{\et})}{H^{i-1}_{L}(Y,\Z/n(j))}\longrightarrow \frac{H^{2c-p+i-1}_{c-p+i-j-2,nr}(X\times Y,(\Z/n)_{X\times Y}(c+j)_{\et})}{H^{2c-p+i-1}_{L}(X\times Y,\Z/n(c+j))}
\end{equation}
induced from \eqref{eq:pairing}. 
\end{corollary}
\begin{proof} This is an immediate consequence of Lemma \ref{lem:pairings-comparison}.\end{proof}

\begin{remark}\label{rem:pairings-comparison-Z} An argument in the same line as the one made in Lemma \ref{lem:pairings-comparison} shows that the pairing 
\begin{equation*}\label{eq:pairing-refined-Z}
\CH^{c}(X,p)\otimes\frac{H^{i-2}_{i-j-3,nr}(Y,(\Q/\Z)_{Y}(j)_{\et})}{H^{i-2}_{L}(Y,\Z(j))}\longrightarrow \frac{H^{2c-p+i-2}_{c-p+i-j-3,nr}(X\times Y,(\Q/\Z)_{X\times Y}(c+j)_{\et})}{H^{2c-p+i-2}_{L}(X\times Y,\Z(c+j))}
\end{equation*}
given by \eqref{eq:pairing} identifies with the natural exterior product map 
\begin{equation*}
    \CH^{c}(X,p)\otimes G^{j,i}(Y)\longrightarrow G^{c+j,2c-p+i}(X\times Y),
\end{equation*}
where the cycle group $G^{j,i}$ is defined in \eqref{eq:G^c,p}.
\end{remark}

\section{Injectivity results}\label{sec:inj}
In this section we prove the main technical results of this paper. The following results and their proofs should be compared with \cite[Theorem 6.1]{Sch-griffiths}.
\begin{theorem}\label{thm:injectivity-CH^d(X,1)} Let $k$ be an algebraically closed field and let $R=\mathcal{O}_{C,x_{0}}$ be the local ring of a smooth $k$-curve $C$ at a point $x_{0}\in C(k)$. Let $K$ denote an algebraic closure of the fraction field $\Frac(R)$ and let $n\geq2$ be an integer invertible in $k$. Assume that there is a projective flat $R$-family $\mathcal{X}\to\Spec R$ with connected fibres of relative dimension $d$ whose total space $\mathcal{X}$ is integral and its generic fibre $X_{\eta}:=\mathcal{X}\times_{R}K$ is smooth and such that the following properties are satisfied, where $X:=\mathcal{X}_{\bar{\eta}}$ is the geometric generic fibre of the family:
\begin{enumerate}[label=$(\textbf{P\arabic*})$]
    \item\label{it:P1} There exist classes $\alpha_{j}\in H^{1}(\mathcal{X}_{\et},\Z/n),$ $j=1,2,\ldots,r,$ such that the elements $\delta(\alpha_{j}|_{X})\in H^{2}_{\cont}(X,\Z_{n})[n]$ form generators, where $\delta: H^{1}(X,\Z/n)\to H^{2}_{\cont}(X,\Z_{n})$ is the Bockstein map and such that
    \item\label{it:P2} for each component $X_{0i}$ of the central fibre $X_{0}=\mathcal{X}\times_{R}k$, the restriction $\alpha_{j}|_{X_{0i}}\in H^{1}(X_{0i},\Z/n)$ is zero. 
\end{enumerate}
 Then for any smooth projective variety $Y_{k}$ over $k$ with base change $Y:=Y_{k}\times_{k}K$ and for any free $\Z/n$-submodule $M\subset\CH^{d}(X,1)/n,$ the exterior product maps
 \begin{equation*}\begin{split}
     &\Lambda_{1}:M\otimes \frac{H^{i}_{r,nr}(Y_{k},\Z/n(j))}{H^{i}(Y_{k},\Z/n(j))} \longrightarrow M\otimes \frac{H^{i}_{r,nr}(Y,\Z/n(j))}{H^{i}(Y,\Z/n(j))}\overset{\eqref{eq:pairing}}{\longrightarrow} \frac{H^{2d+i-1}_{r+d-1,nr}(X\times Y,\Z/n(d+j))}{H^{2d+i-1}(X\times Y,\Z/n(d+j))}\\ 
     & \Lambda_{2}:M\otimes H^{i}_{r,nr}(Y_{k},\Z/n(j)) \longrightarrow M\otimes H^{i}_{r,nr}(Y,\Z/n(j))\overset{\eqref{eq:pairing}}{\longrightarrow} H^{2d+i-1}_{r+d-1,nr}(X\times Y,\Z/n(d+j))
     \end{split}
 \end{equation*}
 are both injective for all integers $i,r\geq0$.
\end{theorem}
\begin{proof} We first prove that $\Lambda_{1}$ is injective. We can write $M=\bigoplus_{\nu=1}^{N}\Z/n\cdot z_{\nu}$, where the classes $z_{\nu}\in \CH^{d}(X,1)_{\tors}$ can be chosen to be torsion; see Lemma \ref{lem:CH^{d}(X,1)/n}. Let $c:=\sum^{N}_{\nu=1}z_{\nu}\otimes[\beta_{k,\nu}]\in\ker (\Lambda_{1}),$ where $\beta_{k,\nu}\in H^{i}(F_{r+1}Y_{k},\Z/n(j))$. We can then find an equi-dimensional closed subscheme $R_{k}\subset Y_{k}$ of codimension-$(r+2)$ in $Y_{k}$, such that $\beta_{k,\nu}\in H^{i}(Y_{k}\setminus R_{k},\Z/n(j))$ for all $\nu=1,2,\ldots,N$. Moreover let $Z\subset X$ be an equi-dimensional closed subscheme of $X$ of codimension-$(d-1)$ in $X$ such that $\supp(z_{\nu})\subset Z\times\Delta^{1}$ for all $\nu=1,2,\ldots,N$. We denote by $\beta_{\nu}$ the image of $\beta_{k,\nu}$ in $H^{i}(Y\setminus R,\Z/n(j)),$ where $R:=R_{k}\times_{k}K$. Then $\Lambda_{1}(c)$ is given as the class of the element 
$$\gamma:=\sum^{N}_{\nu=1} \iota_{\ast}\Exc (p^{\ast}\cl_{Z}(z_{\nu})\cup q^{\ast}\beta_{\nu})\in H^{2d+i-1}(X\times Y\setminus (Z\times R),\Z/n(j+d));$$
see proof of Lemma \ref{lem:pairing}, where $p:X\times Y\to X$ and $q:X\times Y\to Y$ are the natural projections. Note that if we further restrict the class $\gamma$ to $X\times U,$ where $U:=Y\setminus R$ (respectively, we keep the notation $U_{k}:=Y_{k}\setminus R_{k}$), then the class takes the form
$$\gamma|_{X\times U}:=\sum^{N}_{\nu=1} p^{\ast}\cl_{X}(z_{\nu})\cup q^{\ast}\beta_{\nu}\in H^{2d+i-1}(X\times U,\Z/n(j+d)).$$
Now, $$\Lambda_{1}(c)=0\iff [\gamma]\in\im(H^{2d+i-1}(X\times Y, \Z/n(j+d)))\longrightarrow H^{2d+i-1}_{r+d-1,nr}(X\times Y, \Z/n(j+d))).$$ The latter relation also implies that \begin{equation}\label{eq:gamma}[\gamma|_{X\times U}]\in\im(H^{2d+i-1}(X\times Y, \Z/n(j+d))\longrightarrow H^{2d+i-1}_{r+d-1,nr}(X\times U, \Z/n(j+d))).\end{equation} 
The classical Gysin sequence \cite[$\S$2.3.2]{brauer-group} induces a long exact sequence 
\begin{equation}\label{eq:gysin-refined}\begin{split}&\cdots\overset{\partial}{\longrightarrow}\bigoplus_{w\in(X\times U)^{(r+d)}}H^{i-1-2r}(F_{0}\overline{\{\omega\}},\Z/n(j-r))\overset{\iota_{\ast}}{\longrightarrow} H^{2d+i-1}(F_{r+d}(X\times U), \Z/n(j+d))\\ &\longrightarrow H^{2d+i-1}(F_{r+d-1}(X\times U), \Z/n(j+d))\overset{\partial}{\longrightarrow}\cdots;\end{split}\end{equation}
see \cite[Lemma 5.8]{Sch-refined}. Thus, condition \eqref{eq:gamma} together with the exact sequence \eqref{eq:gysin-refined} show that there exists $\xi\in \bigoplus_{w\in(X\times U)^{(r+d)}}H^{i-1-2r}(F_{0}\overline{\{\omega\}},\Z/n(j-r)),$ such that \begin{equation}\label{eq:epsilon}\epsilon:=\iota_{\ast}\xi+\gamma|_{X\times U}\in\im(H^{2d+i-1}(X\times Y, \Z/n(j+d)))\longrightarrow H^{2d+i-1}(F_{r+d}(X\times U), \Z/n(j+d))).\end{equation}
The cycle class map $\cl_{X}:\CH^{d}(X,1)/n\hookrightarrow H^{2d-1}(X_{\et},\Z/n(d))$ is injective; see Lemma \ref{lem:bound_mod_n} \eqref{it:bound_mod_n}. Thus, the classes $\cl_{X}(z_{\nu})$ with $\nu=1,2,\ldots,N$ generate a free $\Z/n$-submodule of $H^{2d-1}(X_{\et},\Z/n(d))$ of rank $N$. We claim the following.
\begin{claim}\label{claim:dual-class} For any $\nu_{0}\in\{1,2,\ldots, N\}$, there exists a class $\overline{\cl_{X}(z_{\nu_{0}})}\in H^{1}(X_{\et},\Z/n)$ with 
\begin{equation}\label{eq:dual-class}
    \overline{\cl_{X}(z_{\nu_{0}})}\cup \cl_{X}(z_{\nu})=\begin{cases}
        \cl^{d}_{X}(pt),\ \text{if}\ \nu=\nu_{0}\\ 
        0, \text{if}\ \nu\neq\nu_{0}.
    \end{cases}
\end{equation}
Moreover, the property \eqref{eq:dual-class} still holds if we add to $\overline{\cl_{X}(z_{\nu_{0}})}$ a class from $H^{1}_{\cont}(X,\Z_n)$, i.e., it only depends on $\delta(\overline{\cl_{X}(z_{\nu_{0}})})\in H^{2}_{\cont}(X,\Z_{n})[n]$.
\end{claim}
\begin{proof}[Proof of Claim \ref{claim:dual-class}] The
 Poincar\'e duality theorem for $X$ yields a perfect pairing
 $$H^{1}(X,\Z/n)\otimes H^{2d-1}(X,\Z/n(d))\longrightarrow H^{2d}(X,\Z/n(d))=\Z/n\cdot\cl^{d}_{X}(pt);$$
 see \cite[Chapter VI, Corollary 11.2]{milne}. In addition, the classes $\cl_{X}(z_{\nu_{0}})$ are $\Z/n$ linearly independent. This in turn implies the existence of $\overline{\cl_{X}(z_{\nu_{0}})}$. For the last claim, recall that $z_{\nu}\in\CH^{d}(X,1)_{\tors}$ is torsion. It follows that the cup product of the torsion class $\cl_{X}(z_{\nu})\in H^{2d-1}_{\cont}(X,\Z_{n}(d))$ with any class of $H^{1}_{\cont}(X,\Z_{n})$ is zero, as $H^{2d}_{\cont}(X,\Z_{n}(d))=\Z_{n}\cdot\cl_{X}(pt)$. Hence, the property \eqref{eq:dual-class} remains unchanged if we add to $\overline{\cl_{X}(z_{\nu_{0}})}$ a class from $H^{1}_{\cont}(X,\Z_n),$ as we wanted.
\end{proof}

 We compute the class $p^{\ast}\overline{\cl_{X}(z_{\nu_{0}})}\cup\epsilon$ with the use of property \eqref{eq:dual-class} and then find from the relation \eqref{eq:epsilon} that \begin{equation*} p^{\ast}\overline{\cl_{X}(z_{\nu_{0}})}\cup\iota_{\ast}\xi+p^{\ast}\cl^{d}_{X}(pt)\cup q^{\ast}\beta_{\nu_{0}}\in\im(H^{2d+i}(X\times Y, \Z/n(j+d))\longrightarrow H^{2d+i}(F_{r+d}(X\times U), \Z/n(j+d))).\end{equation*} We further take pushforward of the above with respect to the proper map $q:X\times Y\to Y$ and obtain that \begin{equation}\label{eq:q_{ast}relation}
    q_{\ast}(p^{\ast}\overline{\cl_{X}(z_{\nu_{0}})}\cup\iota_{\ast}\xi)+\beta_{\nu_{0}}\in\im(H^{i}(Y,\Z/n(j))\longrightarrow H^{i}(F_{r}U,\Z/n(j))).
\end{equation}

\begin{claim} \label{claim:support-xi} We may assume that $\xi=\xi_{1}+\cdots+\xi_{\lambda},$ where each $\xi_{t}\in H^{i-1-2r}(F_{0}\overline{\{\omega_{t}\}},\Z/n(j-r)),$ $t\in\{1,2,\ldots,\lambda\}$ is supported at a point $\omega_{t}\in (X\times U)^{(r+d)}$ with $q(\omega_{t})\in U^{(r)}$.
\end{claim}
\begin{proof}[Proof of Claim \ref{claim:support-xi}] We may write by additivity that $q_{\ast}(p^{\ast}\overline{\cl_{X}(z_{\nu_{0}})}\cup\iota_{\ast}\xi)=\sum_{t=1}^{\mu}q_{\ast}(p^{\ast}\overline{\cl_{X}(z_{\nu_{0}})}\cup\iota_{t\ast}\xi_{t}),$ where each $\xi_{t}\in H^{i-1-2r}(F_{0}\overline{\{\omega_{t}\}},\Z/n(j-r)),$ with $\omega_{t}\in (X\times U)^{(r+d)}$ and where $\iota_{t}:\overline{\{\omega_{t}\}}\hookrightarrow X\times U$ is the natural embedding. It suffices to show that the term $q_{\ast}(p^{\ast}\overline{\cl_{X}(z_{\nu_{0}})}\cup\iota_{t\ast}\xi_{t})$ is zero if $\codim_{Y}(q(\omega_{t}))>r$. Indeed, the projection formula gives that $q_{\ast}(p^{\ast}\overline{\cl_{X}(z_{\nu_{0}})}\cup\iota_{t\ast}\xi_{t})=(q\circ\iota_{t})_{\ast}((p\circ\iota_{t})^{\ast}\overline{\cl_{X}(z_{\nu_{0}})}\cup \xi_{t})$. Thus, the condition $\codim_{Y}(q(\omega_{t}))>r,$ in turn implies that $$(q\circ\iota_{t})_{\ast}((p\circ\iota_{t})^{\ast}\overline{\cl_{X}(z_{\nu_{0}})}|_{F_{0}\overline{\{p(\omega_{t})\}}}\cup \xi_{t})=0,$$ as we wanted.\end{proof}

 Next, we show that the obstruction term $q_{\ast}(p^{\ast}\overline{\cl_{X}(z_{\nu_{0}})}\cup\iota_{\ast}\xi)$ vanishes after specialization. To this end, consider the constant family $\mathcal{U}:=U_{k}\times_{k}R\to\Spec R$ and recall that we have a specialization map $$\bar{sp}: H^{i}(F_{r}U,\Z/n(j))\longrightarrow H^{i}(F_{r}U_{k},\Z/n(j));$$
 see $\S$\ \ref{subsec:specialization_maps} and Proposition \ref{prop:specialization} \eqref{item:sp-geometric}. We want to prove the following claim.
 \begin{claim}\label{claim:sp-obstruction-zero} We have that $\bar{sp}((q\circ\iota)_{\ast}((p\circ\iota)^{\ast}\overline{\cl_{X}(z_{\nu_{0}})}\cup \xi))=0$.
\end{claim}
\begin{proof}[Proof of Claim \ref{claim:sp-obstruction-zero}] First, by linearity and Claim \ref{claim:support-xi}, we may assume that $\xi\in H^{i-1-2r}(F_{0}\overline{\{\omega\}},\Z/n(j-r))$ for some $\omega\in (X\times U)^{(r+d)}$ with $q(\omega)\in U^{(r)}$. We set $W:=\overline{\{\omega\}}\subset X\times U$ and pick a finite field extension $F/\Frac(R)$ over which $W$ is defined, i.e., $W=W_{F}\times_{F}K$ with $W_{F}\subset X_{\eta}\times_{\Frac(R)} \mathcal{U}_{\eta}\times_{\Frac(R)} F$. Up to enlarging $F,$ if necessary, we may also assume that $\xi$ has a lift $\xi_{F}\in H^{i-1-2r}(F_{0}W_{F},\Z/n(j-r))$; see \cite[Chapter III, Lemma 1.16, Remark 1.17 (a)]{milne}. Letting $\tilde{R}$ be a local ring at a maximal ideal of the normalization of $R$ in $F$ and performing the finite ramified base change $\tilde{R}/R,$ we see that the properties \ref{it:P1} and \ref{it:P2} still remain true for the family $\mathcal{X}_{\tilde{R}}\to\Spec\tilde{R}$. This is clear as both maps $H^{1}(\mathcal{X},\Z/n)\to H^{1}(X,\Z/n)$ and $H^{1}(\mathcal{X},\Z/n)\to H^{1}(X_{0i},\Z/n)$ factor through the pullback map $H^{1}(\mathcal{X},\Z/n)\to H^{1}(\mathcal{X}_{\tilde{R}},\Z/n)$. In particular, the above discussion tells us that we can take $F$ as $\Frac(R),$ i.e., $W$ is defined over the generic point of $R$ and that $\xi$ has a lift $$\xi_{\Frac(R)}\in H^{i-1-2r}(F_{0}W_{\Frac(R)},\Z/n(j-r)).$$

\par Next, let $\mathcal{W}\subset \mathcal{X}\times_{R}\mathcal{U}\to\Spec R$ be the schematic closure of $W_{\Frac(R)}\subset X_{\eta}\times_{\Frac(R)}U_{\eta}$ inside $\mathcal{X}\times_{R}\mathcal{U}$ and note that this is a flat $R$-family. It follows from \cite[Theorem 2.5, Corollary 2.2]{branchvarieties} that after a finite ramified base change of discrete valuation rings $R'/R$ followed by normalizing the total space, we can achieve the result that the central fiber of the normalization $\widetilde{\mathcal{W}_{R'}}\to\Spec R'$ is reduced. Thus, up to replacing $R'$ with $R,$ we can assume that the central fibre of the normalization $f:\widetilde{\mathcal{W}}\to\Spec R$ of $\mathcal{W}$ is reduced.

\par The singular locus $\Sing(f)\subset\widetilde{\mathcal{W}}$ of $f$ in turn is a closed subscheme of codimension at least $2$ in $\widetilde{\mathcal{W}}$. Therefore, if we set $\mathcal{S}:=q(\Sing(f))\subset\mathcal{U}$, then for the family $\mathcal{U}\setminus\mathcal{S}\to\Spec R$ one has $F_{r}(\mathcal{U}_{\eta}\setminus\mathcal{S}_{\eta})=F_{r}(\mathcal{U}_{\eta})$ along with $F_{r}(\mathcal{U}_{0}\setminus\mathcal{S}_{0})=F_{r}(\mathcal{U}_{0})$. Especially, we get a proper $R$-morphism $\tilde{q}:\mathcal{W}':=\widetilde{\mathcal{W}}\setminus q^{-1}(\mathcal{S})\to\mathcal{U}':=\mathcal{U}\setminus\mathcal{S},$ whose both domain and target are given by smooth $R$-families with equi-dimensional fibres. We thus obtain a commutative diagram of the form 
\begin{equation}
    \begin{tikzcd}
{H^{i-2r}(F_{0}\widetilde{\mathcal{W}}_{\eta},\Z/n(j-r))} \arrow[r, "sp"] \arrow[d, "(\tilde{q}_{\eta})_{\ast}"] & {H^{i-2r}(F_{0}\mathcal{W}'_{0},\Z/n(j-r))} \arrow[d, "(\tilde{q}_{0})_{\ast}"] \\
{H^{i}(F_{r}\mathcal{U}_{\eta},\Z/n(j))} \arrow[r, "sp"]                                                     & {H^{i}(F_{r}U_{k},\Z/n(j))},           \end{tikzcd}
\end{equation}
where $\tilde{q}_{\eta}$ and $\tilde{q}_{0}$ are the base changes of $\tilde{q}$ over the generic and central fibre, respectively; see Proposition \ref{prop:specialization} \eqref{item:sp-pushforward}. Note that here we drop the symbol $\pi$ from the notation of the specialization maps due to Proposition \ref{prop:specialization} \eqref{item:sp-k-alg}; see Remark \ref{rem:sp-alg}. We also denote by $\tilde{p}:\mathcal{W}'\to\mathcal{X}$ the obvious map. We may assume that up to adding a class from $H^1_{\cont}(X,\Z_n),$ if necessary, the class $\overline{\cl_{X}(z_{\nu_{0}})}\in H^{1}(X,\Z/n)$ has a lift $\alpha_{\nu_{0}}\in H^{1}(\mathcal{X},\mu_{n})$ that satisfies the property \ref{it:P2}; see Claim \ref{claim:dual-class}. Then by construction of $\bar{sp},$ we have the equality 
\begin{equation*}\begin{split}\bar{sp}((q\circ\iota)_{\ast}((p\circ\iota)^{\ast}\overline{\cl_{X}(z_{\nu_{0}})}\cup \xi))&=sp((\tilde{q}_{\eta})_{\ast}(((\tilde{p}_{\eta})^{\ast}\alpha_{\nu_{0}})|_{F_{0}\widetilde{\mathcal{W}}_{\eta}}\cup \xi_{\Frac(R)}))\\& = (\tilde{q}_{0})_{\ast}(sp(((\tilde{p}_{\eta})^{\ast}\alpha_{\nu_{0}})|_{F_{0}\widetilde{\mathcal{W}}_{\eta}}\cup \xi_{\Frac(R)}))\in H^{i}(F_{r}U_{k},\Z/n(j)).\end{split}\end{equation*}
By Proposition \ref{prop:specialization} \eqref{item:sp-global-class}, we also find that $sp(((\tilde{p}_{\eta})^{\ast}\alpha_{\nu_{0}})|_{F_{0}\widetilde{\mathcal{W}}_{\eta}}\cup \xi_{\Frac(R)})=\tilde{p}
^{\ast}\alpha_{\nu_{0}}|_{F_{0}\mathcal{W}_{0}'}\cup sp(\xi_{\Frac(R)})=0,$ where the last term is zero, since $\tilde{p}
^{\ast}\alpha_{\nu_{0}}|_{F_{0}\mathcal{W}_{0}'}=0$ by \ref{it:P2}. This finishes the proof of the claim.
\end{proof}
Finally, note that $\bar{sp}(\beta_{\nu_{0}})=\beta_{k,\nu_{0}}$; see Proposition \ref{prop:specialization} \eqref{item:sp-global-class}. Thus, by \eqref{eq:q_{ast}relation}, we see that Claim \ref{claim:sp-obstruction-zero} together with Proposition \ref{prop:specialization} \eqref{item:sp-F_*} imply that $\beta_{k,\nu_{0}}\in \im(H^{i}(Y_{k},\Z/n(j))\to H^{i}(F_{r}Y_{k},\Z/n(j))),$ where we use that $F_{r}U_{k}=F_{r}Y_{k}$. This proves the injectivity of $\Lambda_{1}$.
\par We now explain how one can deduce the injectivity of $\Lambda_{2}$. With the same notation as above, we see that $\Lambda_{2}(c)=0$ implies $[\gamma|_{X\times U}]=0\in H^{2d+i-1}_{r+d-1,nr}(X\times U,\Z/n(j+d))$. Thus, condition \eqref{eq:epsilon} becomes $\gamma|_{X\times U}+\iota_{\ast}\xi=0\in H^{2d+i-1}(F_{r+d}(X\times U),\Z/n(j+d))$ and instead of \eqref{eq:q_{ast}relation}, one in turn gets $ q_{\ast}(p^{\ast}\overline{\cl_{X}(z_{\nu_{0}})}\cup\iota_{\ast}\xi)+\beta_{\nu_{0}}=0\in H^{i}(F_{r}U,\Z/n(j))$. Hence, as before Claim \ref{claim:sp-obstruction-zero} along with Proposition \ref{prop:specialization} \eqref{item:sp-global-class} yield that $$0=\bar{sp}(q_{\ast}(p^{\ast}\overline{\cl_{X}(z_{\nu_{0}})}\cup\iota_{\ast}\xi)+\beta_{\nu_{0}})=\bar{sp}(\beta_{\nu_{0}})=\beta_{k,\nu_{0}}\in H^{i}(F_{r}Y_{k},\Z/n(j)),$$ as we wanted. This proves the injectivity of $\Lambda_{2}$. The proof is complete.
\end{proof}

\begin{remark}\label{remark:P1-P2} If one fixes the free $\Z/n$-submodule $M=\bigoplus_{\nu=1}^{N}\Z/n\cdot z_{\nu}\subset\CH^{d}(X,1)/n$ in Theorem \ref{thm:injectivity-CH^d(X,1)}, then both properties \ref{it:P1} and \ref{it:P2} can be relaxed further. In fact, for \ref{it:P1}, we only need that $$\delta(\overline{\cl_{X}(z_{\nu})})\in \im( H^{1}(\mathcal{X},\Z/n)\to H^{1}(X,\Z/n)\overset{\delta}{\to} H^{2}_{\cont}(X,\Z_{n})),$$ where $\overline{\cl_{X}(z_{\nu})}\in H^{1}(X,\Z/n)$ is a Poincar\'e dual of the class $\cl_{X}(z_{\nu})\in H^{2d-1}(X,\Z/n(d))$; see Claim \ref{claim:dual-class}. In addition, for \ref{it:P2}, it suffices to require that the lifts of $\delta(\overline{\cl_{X}(z_{\nu})})$ in $H^{1}(\mathcal{X},\Z/n)$ can be chosen in such a way that they become zero when restricted to the components $X_{0i}$ of the central fibre $X_{0}:=\mathcal{X}\times_{R}k$.
\end{remark}

\begin{theorem}\label{thm:injectivity-CH^d(X,1;Z/n)}Let $k$ be an algebraically closed field and let $R=\mathcal{O}_{C,x_{0}}$ be the local ring of a smooth $k$-curve $C$ at a point $x_{0}\in C(k)$. Let $K$ denote an algebraic closure of the fraction field $\Frac(R)$ and let $n\geq2$ be an integer invertible in $k$. Assume that there is a projective flat $R$-family $\mathcal{X}\to\Spec R$ with connected fibres of relative dimension $d$ whose total space $\mathcal{X}$ is integral and its generic fibre $X_{\eta}:=\mathcal{X}\times_{R}K$ is smooth and such that the following properties are satisfied, where $X:=\mathcal{X}_{\bar{\eta}}$ is the geometric generic fibre of the family:
\begin{enumerate}[label=$(\textbf{P'\arabic*})$]
    \item\label{it:P'1} There exist classes $\alpha_{j}\in H^{1}(\mathcal{X}_{\et},\Z/n),$ $j=1,2,\ldots,r,$ such that the elements $\alpha_{j}|_{X}\in H^{1}(X,\Z/n)$ form generators and such that
    \item\label{it:P'2} for each component $X_{0i}$ of the central fibre $X_{0}=\mathcal{X}\times_{R}k$, the restriction $\alpha_{j}|_{X_{0i}}\in H^{1}(X_{0i},\Z/n)$ is zero. 
\end{enumerate}
 Then for any smooth projective variety $Y_{k}$ over $k$ with base change $Y:=Y_{k}\times_{k}K$ and any free $\Z/n$-submodule $M\subset\CH^{d}(X,1\ ;\Z/n),$ the exterior product maps
 \begin{equation*}\begin{split}
     &\Lambda_{1}:M\otimes \frac{H^{i}_{r,nr}(Y_{k},\Z/n(j))}{H^{i}(Y_{k},\Z/n(j))} \longrightarrow M\otimes \frac{H^{i}_{r,nr}(Y,\Z/n(j))}{H^{i}(Y,\Z/n(j))}\overset{\eqref{eq:pairing}}{\longrightarrow} \frac{H^{2d+i-1}_{r+d-1,nr}(X\times Y,\Z/n(d+j))}{H^{2d+i-1}(X\times Y,\Z/n(d+j))}\\ 
     & \Lambda_{2}:M\otimes H^{i}_{r,nr}(Y_{k},\Z/n(j)) \longrightarrow M\otimes H^{i}_{r,nr}(Y,\Z/n(j))\overset{\eqref{eq:pairing}}{\longrightarrow} H^{2d+i-1}_{r+d-1,nr}(X\times Y,\Z/n(d+j))
     \end{split}
 \end{equation*}
 are both injective for all integers $i,r\geq0$.
\end{theorem}
\begin{proof} Recall that we have a canonical isomorphism $\cl_{X}:\CH^{d}(X,1\ ;\Z/n)\cong H^{2d-1}(X_{\et},\Z/n(d));$ see Lemma \ref{lem:bound_mod_n} \eqref{it:bound_mod_n}. Hence, if $M:=\bigoplus^{N}_{\nu=1}\Z/n\cdot z_{\nu}$, then $\cl_{X}(M)\subset H^{2d-1}(X_{\et},\Z/n(d))$ is a free $\Z/n$-submodule of rank $N$ generated by the classes $\cl_{X}(z_{\nu})$ with $\nu=1,2,\ldots, N$.\par
Let us explain why we need the slightly stronger condition \ref{it:P'1} here. This is related to Claim \ref{claim:dual-class}. Namely, the cup product $\overline{\cl_{X}(z_{\nu_{0}})}\cup\cl_{X}(z_{\nu}),$ where $\overline{\cl_{X}(z_{\nu_{0}})}\in H^{1}(X,\Z/n)$ is a Poincar\'e dual of the class $\cl_{X}(z_{\nu_{0}})\in H^{2d-1}(X_{\et},\Z/n(d)),$ (see \eqref{eq:dual-class}) might change if we add to $\overline{\cl_{X}(z_{\nu_{0}})}$ a class that lies in the image of $H^{1}_{\cont}(X,\Z_{n})\to H^{1}(X,\Z/n)$. Indeed, if $z_{\nu_{0}}\notin\CH^{d}(X,1)/n,$ then the class $\cl_{X}(z_{\nu_{0}})\in H^{2d-1}(X,\Z/n(d))$ no longer admits a lift in $H^{2d-1}_{\cont}(X,\Z_{n}(d))_{\tors}$; see Lemma \ref{lem:CH^{d}(X,1)/n}. The latter can of course occur if and only if the first Betti number $b_{1}$ \cite[Remark 5.2.7]{brauer-group} of $X$ is non-zero by Roitman's theorem \cite{bloch-roitman,roitman} and the following exact sequence 
$$0\longrightarrow\CH^{d}(X,1)/n\longrightarrow \CH^{d}(X,1\ ;\Z/n) \overset{\delta}{\longrightarrow}\CH_{0}(X)[n]\longrightarrow 0,$$ which is induced from the long exact coefficient sequence also known as Bockstein sequence.\par
Up to the above clarification, the proof of Theorem \ref{thm:injectivity-CH^d(X,1;Z/n)} is identical with the one of Theorem \ref{thm:injectivity-CH^d(X,1)}.
\end{proof}
\begin{remark}\label{remark:P'1-P'2} As explained in Remark \ref{remark:P1-P2}, the properties \ref{it:P'1} and \ref{it:P'2} in Theorem \ref{thm:injectivity-CH^d(X,1;Z/n)} can also be relaxed if one fixes the free $\Z/n$-submodule $M:=\bigoplus_{\nu=1}^{N}\Z/n\cdot z_{\nu}\subset\CH^{d}(X,1\ ;\Z/n)$. That is, for \ref{it:P'1}, it is enough to require that 
$$\overline{\cl_{X}(z_{\nu})}\in \im( H^{1}(\mathcal{X},\Z/n)\to H^{1}(X,\Z/n)),$$ where $\overline{\cl_{X}(z_{\nu})}\in H^{1}(X,\Z/n)$ is a Poincar\'e dual of the class $\cl_{X}(z_{\nu})\in H^{2d-1}(X,\Z/n(d))$; see Claim \ref{claim:dual-class}. Moreover, for \ref{it:P'2}, it suffices to require that one can choose the lifts of $\overline{\cl_{X}(z_{\nu})}$ in $H^{1}(\mathcal{X},\Z/n),$ in such a way that they become zero when restricted to the components $X_{0i}$ of the central fibre $X_{0}:=\mathcal{X}\times_{R}k$.\end{remark}

\begin{remark}\label{remark:Intuition-for-P1-P2} Recall that if $X$ is any scheme and $G$ is a finite group, then the \'etale cohomology group $H^{1}(X_{\et},G)$ classifies Galois coverings of $X$ with group $G$; see \cite[Chapter III, $\S$4]{milne}. Thus, the condition \ref{it:P'1} says that $\Z/n$-Galois coverings of $X$ lift as $\Z/n$-Galois coverings across the total space $\mathcal{X}$ and \ref{it:P'2} in turn implies that these coverings become trivial over the components $X_{0i}$ of the central fibre $X_{0}$.
\end{remark}

\section{Construction}\label{sec:construction}
In this section we introduce the surfaces together with their degenerations which will ultimately satisfy the properties \ref{it:P1} and \ref{it:P2} of Theorem \ref{thm:injectivity-CH^d(X,1)}. The surfaces we consider have already seen applications in \cite{Ale-griffiths}.\par
Let $k$ be an algebraically closed field. Let $n\geq2$ be an integer invertible in $k$ and let $\zeta\in k$ be a primitive $n$-th root of unity. As in \cite[$\S$3]{Ale-griffiths}, we consider the automorphism $\varphi$ of $\mathbb{P}^{5}_{k}$ 
\begin{equation}\label{eq:varphi}
    (x_{0}:x_{1}:x_{2}:x_{3}:x_{4}:x_{5})\mapsto (x_{0}:\zeta x_{1}:x_{2}:\zeta x_{3}:x_{4}:\zeta x_{5})
\end{equation}
whose fixed locus $\Fix(\varphi)$ is the union of the two planes $$P_{1}:=\{x_{0}=x_{2}=x_{4}=0\}\ \text{and}\ P_{2}:=\{x_{1}=x_{3}=x_{5}=0\}.$$ 
We let $\varphi$ act on the linear system $|\mathcal{O}_{\mathbb{P}^{5}_{k}}(n)|$ which parameterizes hypersurfaces of $\mathbb{P}^{5}_{k}$ of degree $n$ and then consider its sublinear system $|\mathcal{O}_{\mathbb{P}^{5}_{k}}(n)|^{\varphi}$ which is defined as the eigenspace of the eigenvalue $\lambda=1$. This in turn is a base-point-free linear system of dimension $d=(n+1)(n+2)-1$.\par
For general choices of invariant hypersurfaces $H_{i}\in |\mathcal{O}_{\mathbb{P}^{5}_{k}}(n)|^{\varphi},$ with $i\in\{1,2,3\}$, the complete intersection $T:=H_{1}\cap H_{2}\cap H_{3}$ is smooth and disjoint from $\Fix(\varphi)$. Hence, the quotient $S:=T/\varphi$ is also a smooth projective surface. Note that the first Betti number $b_{1}:=\dim_{\Q_{\ell}}H^{1}_{\cont}(S,\Q_{\ell}(0))$ is zero, where $\ell$ is any prime invertible in $k$ and $\Pic(S)_{\tors}\cong \Z/n$. Moreover, if $n = 2$, then $S$ is an Enriques surface and $T$ is its canonical double cover, that is, a $K3$ surface, while if $n\geq 3,$ then $S$ is a surface of general type; see \cite[Proposition 3.1]{Ale-griffiths}. \par
We include some of the properties of the surface $S$ which will be recalled later.

\begin{proposition}\label{pro:properties_S} Let $k$ be an algebraically closed field and let $n\geq2$ be an integer invertible in $k$. Let $T\subset\mathbb{P}^{5}_{k}$ be a smooth complete intersection of three invariant hypersurfaces $H_{i}\in{|\mathcal{O}_{\mathbb{P}^{5}_{k}}(n)|}^{\varphi}$, such that the automorphism $\varphi$ \eqref{eq:varphi} acts freely on $Y$. Let $S:=T/\varphi$ be the corresponding $\Z/n$-quotient of $T$. Then the following assertions are true:
\begin{enumerate}
     \item\label{it:NS(S)} We have $\Pic(S)=\NS(S)$ and $\NS(S)_{\tors}\cong\Z/n$. 
     \item\label{it:Bockstein iso} The Bockstein map yields an isomorphism $\delta : H^{1}(S_{\et},\Z/n)\cong H^{2}_{\cont}(S,\Z_{n})_{\tors}\cong\Z/n$.
     \item\label{it:CH2(S,1)/n} We have $\CH^{2}(S,1)[n]/n\CH^{2}(S,1)[n^{2}]=\CH^{2}(S,1)/n$ and $\CH^{2}(S,1)/n\cong\Z/n$.
     \item\label{it:hodge number} In addition, if $k=\C,$ then the Hodge numbers of $S$ are given by $h^{2,0}=\frac{n^{2}}{4}(5n^{2}-18n+17)-1$, $h^{1,1}=\frac{n^{2}}{2}(7n^{2}-18n+13),$ and $h^{1,0}=0$.
 \end{enumerate}
\end{proposition}
\begin{proof} The properties are well-known. We include a proof for the sake of completeness. The Hochschild--Serre spectral sequence associated to the $\Z/n$-Galois covering $T\to S,$ $$E^{p,q}_{2}:=H^{p}(\Z/n,H^{q}(T_{\et},\mathbb{G}_{m}))\implies H^{p+q}(S_{\et},\mathbb{G}_{m})$$ yields a short exact sequence 
$$0\longrightarrow H^{1}(\Z/n, H^{0}(T_{\et},\mathbb{G}_{m}))\longrightarrow \Pic(S)\longrightarrow \Pic(T)^{\Z/n}\longrightarrow 0,$$
where we implicitly use the canonical isomorphism $H^{1}(-_{\et},\mathbb{G}_{m})\cong \Pic(-)$; see \cite[Chapter III, Proposition 4.9]{milne}. The first term of the above sequence is identified with the character group of $\Z/n$ and is thus again isomorphic to $\Z/n$. Now, the claim in \eqref{it:NS(S)} follows from the above short exact sequence along with the well-known fact that any smooth complete intersection of hypersurfaces in a projective space has invariants $\Pic^{0}$ and $\Pic_{\tors}$ zero.\par
Note that the claim in \eqref{it:Bockstein iso} is a direct consequence of \eqref{it:NS(S)}. In fact, we have a short exact sequence $$0\longrightarrow H^{1}_{\cont}(S,\Z_{n})/n\longrightarrow H^{1}(S_{\et},\Z/n)\longrightarrow H^{2}_{\cont}(S,\Z_{n})_{\tors}\longrightarrow0,$$ which arises from the long exact Bockstein sequence. The Kummer exact sequence \cite[$\S$3.2.1]{brauer-group} gives an isomorphism $H^{1}(S_{\et},\Z/m)\cong\Pic(S)[m]$ for any invertible integer $m\in k^{\ast}$ and hence \eqref{it:NS(S)} implies $H^{1}_{\cont}(S,\Z_{n})=0$; see \cite[(0.2)]{jannsen}. This proves the claim in \eqref{it:Bockstein iso}.\par

Next, we show the claim in \eqref{it:CH2(S,1)/n}. Note that $H^{3}_{\cont}(S,\Z_{n}(2))\cong\Z/n$, and thus we find that the cycle class map $$\frac{\CH^{2}(S,1)[n]}{n\CH^{2}(S,1)[n^2]}\overset{\cl}{\longrightarrow} H^{3}_{\cont}(S,\Z_{n}(2))_{\tors}/n$$ is an isomorphism; see Lemma \ref{lem:CH^{d}(X,1)/n}. Moreover, the same lemma implies $$\frac{\CH^{2}(S,1)[n]}{n\CH^{2}(S,1)[n^2]}=\CH^2(S,1)_{\tors}/n=\CH^{2}(S,1)/n,$$ as we wanted.\par
Finally, the claim in \eqref{it:hodge number} is contained in \cite[Proposition 3.1]{Ale-griffiths}. The proof of the proposition is complete.
\end{proof}

The following result gives the desired degeneration of the surface $S$.

\begin{theorem}\label{thm:degeneration} Let $k$ be an algebraically closed field and let $K$ be an algebraic closure of the field $k(t)$. Let $n\geq2$ be an integer invertible in $k$. Then there exists a projective flat family $\mathcal{S}\to\Spec \mathcal{O}_{\mathbb{A}^{1}_{k},0},$ such that 
\begin{enumerate}
    \item\label{it:S} The geometric generic fibre $S:=\mathcal{S}\times_{\mathcal{O}_{\mathbb{A}^{1}_{k},0}}K$ is a quotient of the form $T/\varphi,$ where $T\subset \mathbb{P}^{5}_{K}$ is a smooth
complete intersection of three very general invariant hypersurfaces $H_{i}\in |\mathcal{O}_{\mathbb{P}^{5}_{K}}(n)|^{\varphi}$ and $\varphi$ is
the automorphism \eqref{eq:varphi}.
\item\label{it:S_0} The central fibre $S_{0}=\bigcup_{i=1}^{n^{2}}S_{0i}$ is reduced and consists of $n^{2}$ rational components $S_{0i}$, each isomorphic to the projective plane $\mathbb{P}^{2}_{k}$. Moreover, the double curves $C_{i,j}:=S_{0i}\cap S_{0j}$ sit as projective lines inside each plane $S_{0i}$.
\item\label{it:galois-covering} The $\Z/n$-Galois covering $T\to S$ extends to a $\Z/n$-Galois covering $\mathcal{T}\to\mathcal{S}$ over $\Spec \mathcal{O}_{\mathbb{A}^{1}_{k},0}$, which in turn is trivial over each component $S_{0i}$ of the central fibre $S_{0}$.
\end{enumerate}
\end{theorem}
\begin{proof} We first construct the family $\mathcal{T}\to\Spec \mathcal{O}_{\mathbb{A}^{1}_{k},0}$ mentioned in \eqref{it:galois-covering}. For a hyperplane $L\in|\mathcal{O}_{\mathbb{P}^{5}_{k}}(1)|$ we put $L_{\varphi^{i}}:={(\varphi^{i})}^{\ast}(L)\subset\mathbb{P}^{5}_{k}$. We shall use the following lemma.
\begin{lemma}\label{lem:central-fibre} For general $L^{1},L^{2},L^{3}\in |\mathcal{O}_{\mathbb{P}^{5}_{k}}(1)|$ one has that the intersections $L^{1}_{\varphi^{i_{1}}}\cap L^{2}_{\varphi^{i_{2}}}\cap L^{3}_{\varphi^{i_{3}}}$ and $L^{1}_{\varphi^{j_{1}}}\cap L^{2}_{\varphi^{j_{2}}}\cap L^{3}_{\varphi^{j_{3}}}$ are disjoint if  $i_{r}\neq j_{r}$ for all $r\in\{1,2,3\},$ where $0\leq i_{r},j_{r}\leq n-1$.
\end{lemma} 
\begin{proof}[Proof of Lemma \ref{lem:central-fibre}] This can be translated into a determinant criterion that yields an open condition. We include some of the details. Write 
\begin{equation*}
\begin{split} &L^{1}:=\{\alpha_{0}x_{0}+\alpha_{1}x_{1}+\alpha_{2}x_{2}+\alpha_{3}x_{3}+\alpha_{4}x_{4}+\alpha_{5}x_{5}=0\}\\
   &L^{2}:=\{\beta_{0}x_{0}+\beta_{1}x_{1}+\beta_{2}x_{2}+\beta_{3}x_{3}+\beta_{4}x_{4}+\beta_{5}x_{5}=0\}\\
   &L^{3}:=\{\gamma_{0}x_{0}+\gamma_{1}x_{1}+\gamma_{2}x_{2}+\gamma_{3}x_{3}+\gamma_{4}x_{4}+\gamma_{5}x_{5}=0\}
    \end{split}
\end{equation*}
for the equations of $L^{i}$ with $\alpha_{r},\beta_{r},\gamma_{r}\in k$. We let $A\in M(6\times6,k)$ be the matrix whose first 3 rows are given by the coefficients of the equations of $L^{r}_{\varphi^{i_{r}}}$ with $r\in\{1,2,3\}$ and whose last 3 rows are given by the coefficients of the equations of $L^{r}_{\varphi^{j_{r}}}$ with $r\in\{1,2,3\}$, that is,
\begin{equation*}
    A:=\begin{pmatrix}
   \alpha_{0}& \zeta^{i_{1}}\alpha_{1}& \alpha_{2}& \zeta^{i_{1}}\alpha_{3}&\alpha_{4}&\zeta^{i_{1}}\alpha_{5}\\
   \beta_{0}& \zeta^{i_{2}}\beta_{1}& \beta_{2}& \zeta^{i_{2}}\beta_{3}&\beta_{4}&\zeta^{i_{2}}\beta_{5}\\
   \gamma_{0}& \zeta^{i_{3}}\gamma_{1}& \gamma_{2}& \zeta^{i_{3}}\gamma_{3}&\gamma_{4}&\zeta^{i_{3}}\gamma_{5}\\
    \alpha_{0}& \zeta^{j_{1}}\alpha_{1}& \alpha_{2}& \zeta^{j_{1}}\alpha_{3}&\alpha_{4}&\zeta^{j_{1}}\alpha_{5}\\
   \beta_{0}& \zeta^{j_{2}}\beta_{1}& \beta_{2}& \zeta^{j_{2}}\beta_{3}&\beta_{4}&\zeta^{j_{2}}\beta_{5}\\
   \gamma_{0}& \zeta^{j_{3}}\gamma_{1}& \gamma_{2}& \zeta^{j_{3}}\gamma_{3}&\gamma_{4}&\zeta^{j_{3}}\gamma_{5}
    \end{pmatrix}.
\end{equation*}
It can be easily checked that \begin{equation}\label{eq:det-A}\det(A)=c(\zeta^{i_{1}}-\zeta^{j_{1}})(\zeta^{i_{2}}-\zeta^{j_{2}})(\zeta^{i_{3}}-\zeta^{j_{3}})(\alpha_{0}\gamma_{4}\beta_{2}-\alpha_{0}\beta_{4}\gamma_{2}+\beta_{0}\gamma_{2}\alpha_{4}-\beta_{0}\alpha_{2}\gamma_{4}-\gamma_{0}\beta_{2}\alpha_{4}+\gamma_{0}\beta_{4}\alpha_{2}),\end{equation}
where $c:=\det\begin{pmatrix}
    \alpha_{1}&\alpha_{3}&\alpha_{5}\\
    \beta_{1}&\beta_{3}&\beta_{5}\\
    \gamma_{1}&\gamma_{3}&\gamma_{5}
\end{pmatrix}.$\par
The calculation \eqref{eq:det-A} together with the assumption $i_{r}\neq j_{r}$ for all $r\in\{1,2,3\},$ where $0\leq i_{r},j_{r}\leq n-1$ show that the determinant of the matrix $A$ is defined by a non-zero polynomial. Thus, there exists a non-empty open subset $U_{(\bar{i},\bar{j})}\subset |\mathcal{O}_{\mathbb{P}^{5}_{k}}(1)|^{\times 3},$ such that for every point $(L^{1},L^{2},L^{3})\in U_{(\bar{i},\bar{j})},$ the intersections $L^{1}_{\varphi^{i_{1}}}\cap L^{2}_{\varphi^{i_{2}}}\cap L^{3}_{\varphi^{i_{3}}}$ and $L^{1}_{\varphi^{j_{1}}}\cap L^{2}_{\varphi^{j_{2}}}\cap L^{3}_{\varphi^{j_{3}}}$ are disjoint, where $\bar{i}:=(i_{1},i_{2},i_{3})$ and $\bar{j}:=(j_{1},j_{2},j_{3})$. Finally, consider the finite index set $I:=\{(\bar{i},\bar{j})\in\{0,1,2,\ldots,n-1\}^{\times 6}|\ i_{r}\neq j_{r}\}$ and put $U:=\bigcap_{(\bar{i},\bar{j})\in I} U_{(\bar{i},\bar{j})}$. Then $U$ is a non-empty open subset as $|\mathcal{O}_{\mathbb{P}^{5}_{k}}(1)|^{\times 3}$ is irreducible and clearly the closed points of $U$ have the desired property. This proves Lemma \ref{lem:central-fibre}.
\end{proof}
Now, let $H_{i}:=\{F_{i}=0\}\in |\mathcal{O}_{\mathbb{P}^{5}_{k}}(n)|^{\varphi},$ with $i\in\{1,2,3\},$ be general invariant hypersurfaces and choose hyperplanes $L^{1},L^{2},L^{3}\in|\mathcal{O}_{\mathbb{P}^{5}_{k}}(1)|,$ as in Lemma \ref{lem:central-fibre}. We then consider the one parameter family $\mathcal{T}\subset \mathbb{P}^{5}_{k}\times \mathbb{A}^{1}_{k}$ of complete intersections defined by the equations 
\begin{equation}\label{eq:T}
\begin{split}  &(1-t)\prod_{i=0}^{n-1}L^{1}_{\varphi^{i}}+tF_{1}=0,\\
    & (1-t)\prod_{i=0}^{n-1}L^{2}_{\varphi^{i}}+tF_{2}=0,\\
    & (1-t)\prod_{i=0}^{n-1}L^{3}_{\varphi^{i}}+tF_{3}=0,
 \end{split}
\end{equation}
where $t\in \mathbb{A}^{1}_{k}$ and where by slight abuse of notation we keep the same notation for the equations of $L^{i}_{\varphi^{j}}$. The fibre over $t=0$ then consists of $n^{3}$ planes $T_{0i}$. The $\Z/n$-orbit of each component $T_{0i}$ is a disjoint union of its $n$ conjugates and moreover the double intersection $D_{i}:=\bigcup_{j\neq i} (T_{0i}\cap T_{0j})$ consists of $3(n-1)$ lines that sit inside the plane $T_{0i}$.\par
Next, we perform the base change $\Spec\mathcal{O}_{\mathbb{A}^{1}_{k},0}\to\mathbb A^{1}_{k}$ and obtain a projective flat family $\mathcal{T}\to\Spec \mathcal{O}_{\mathbb{A}^{1}_{k},0}$, which is invariant under the $\Z/n$-action and has no fixed points. It is then readily seen that by our construction the family $\mathcal{S}:=\mathcal{T}/(\Z/n)\to \Spec\mathcal{O}_{\mathbb{A}^{1}_{k},0}$ satisfies the properties \eqref{it:S}, \eqref{it:S_0} and \eqref{it:galois-covering}. The proof of Theorem \ref{thm:degeneration} is complete.
\end{proof}

\begin{remark}\label{remark:semi-stable} We note that the family $\mathcal{S}\to\Spec\mathcal{O}_{\mathbb{A}^{1}_{k},0}$ constructed in Theorem \ref{thm:degeneration} is not semi-stable. For instance there are distinct components in the central fibre whose intersection is a point and those one needs to blow them up further. Eventually the resulting semi-stable degeneration will satisfy the properties \eqref{it:S}--\eqref{it:galois-covering} with the exception in \eqref{it:S_0} all components are rational but not necessarily isomorphic to the projective plane and the double intersections $D_{i}:=\bigcup_{j\neq i}(T_{0i}\cap T_{0j})$ will form a cycle of rational curves on each component $T_{0i}$. The case $n=2$ yields a Type III semi-stable degeneration of Enriques; see \cite[Theorem III]{kulikov}. The reason we do not use semi-stable families in this paper are the results of $\S$\ref{sec:inj} that allow us to work with more general families; see Theorems \ref{thm:injectivity-CH^d(X,1)} and \ref{thm:injectivity-CH^d(X,1;Z/n)}.
\end{remark}

\section{Main results}\label{sec:main-results}
\subsection{Examples with infinite morphic cohomology modulo any prime $\ell$}
In this subsection, we prove Theorem \ref{thm:general-totaro-result} and Corollary \ref{cor:general-totaro-result}. These are essential for the proof of Theorem \ref{thm:main}, but they are also of independent interest. We begin with the following elementary lemma.

\begin{lemma}\label{lem:kernel-exterior-product} Let $k$ be an algebraically closed field and let $n\geq2$ be an integer invertible in $k$. Let $X$ and $Y$ be smooth and equi-dimensional algebraic schemes over $k$ and assume that $X$ is also proper. Fix a higher Chow cycle $[\Gamma]\in\CH^{c_{1}}(X,p_{1}\ ;\Z/n),$ such that its cycle class $\cl_{X}([\Gamma])\in H^{2c_{1}-p_{1}}(X_{\et},\Z/n(c_{1}))$ is non-zero. Then the kernel of the exterior product map 
\begin{equation*}
    \CH^{c_{2}}(Y,p_{2}\ ;\Z/n)\longrightarrow \CH^{c_{1}+c_{2}}(X\times Y, p_{1}+p_{2}\ ;\Z/n),[z]\mapsto [\Gamma\times z]
\end{equation*}
 is contained in $F_{n}^{c_{2}, 2c_{2}-p_{2}}(Y):=\ker(\cl_{Y}:\CH^{c_{2}}(Y,p_{2}\ ;\Z/n)\to H^{2c_{2}-p_{2}}(Y_{\et},\Z/n(c_{2})))$.
\end{lemma}
\begin{proof} Indeed, let $[z]\in \CH^{c_{2}}(Y,p_{2}\ ;\Z/n),$ such that $[\Gamma\times z]=0\in\CH^{c_{1}+c_{2}}(X\times Y, p_{1}+p_{2}\ ;\Z/n)$. Then we have the equality $$0=\cl_{X\times Y}([\Gamma\times z])=p^{\ast}\cl_{X}([\Gamma])\cup q^{\ast}\cl_{Y}([z])\in H^{2(c_{1}+c_{2})-(p_{1}+p_{2})}(X\times Y,\Z/n(c_{1}+c_{2})),$$where $p:X\times Y\to X$ and $q:X\times Y\to Y$ are the natural projections. The class $\cl_{X}([\Gamma])\neq0$ is non-zero and thus, we can consider its Poincar\'e dual $$\overline{\cl_{X}([\Gamma])}\in H^{2(d-c_{1})+p_{1}}(X_{\et},\Z/n(d-c_{1})),$$ i.e., $\overline{\cl_{X}([\Gamma])}\cup \cl_{X}([\Gamma])=\cl^{d}_{X}(pt)$ is the class of a point, where $d:=\dim X$; see \cite[Chapter VI, Corollary 11.2]{milne}. We then see that the class $\cl_{Y}([z])$ can be recovered as $$q_{\ast}(p^{\ast}\cl_{X}^{d}(pt)\cup q^{\ast}\cl_{Y}([z]))=q_{\ast}(p^{\ast}\overline{\cl_{X}([\Gamma])}\cup p^{\ast}\cl_{X}([\Gamma])\cup q^{\ast}\cl_{Y}([z])).$$
This in turn implies that $\cl_{Y}([z])=0,$ as we wanted. The proof of the lemma is complete.
\end{proof}

\begin{theorem}\label{thm:ext-prod-elliptic} Let $k$ be an algebraically closed field and let $k\subset K$ be an algebraically closed field extension of positive transcendence degree. Let $E_{K}$ be an elliptic curve over $K$ whose $J$ invariant is transcendental over $k$. Let $n\geq2$ be an integer invertible in $k$. Then for every free $\Z/n$-submodule $M\subset \CH^{1}(E_{K},1\ ; \Z/n)$ of rank one, the following holds: For any smooth projective variety $Y$ over $k$, with base change $Y_{K}:=Y\times_{k}K,$ the exterior
product map
\begin{equation}\label{eq:ext-prod-elliptic}
    M\otimes \CH^{c}(Y,p\ ;\Z/n)\longrightarrow \CH^{c+1}(E_{K}\times Y_{K},p+1\ ;\Z/n),\ [\Gamma]\otimes[z]\mapsto [\Gamma\times z_{K}]
\end{equation}
 is injective for all integers $c,p\geq0$.   
\end{theorem}
\begin{proof} A rigidity theorem of Jannsen implies that it suffices to consider only the case where the field $K$
is an algebraic closure of $k(J)$; see \cite[Theorem 0.3]{jannsen-r}.\par
We want to apply Theorem \ref{thm:injectivity-CH^d(X,1;Z/n)}. Let $M\subset \CH^{1}(E_{K},1\ ; \Z/n)$ be a free $\Z/n$-submodule of rank one. Recall that we have a natural isomorphism $\CH^{1}(E_{K},1\ ; \Z/n)\cong H^{1}(E_{K},\Z/n)$ (see Theorem \ref{thm:Licthenbaum-Beilinson-conjecture} \eqref{it:L-B}) and we let $\tilde{M}\subset H^{1}(E_{K},\Z/n)$ be a $\Z/n$-submodule generated by some Poincar\'e dual of a generator of $M$. Note that $\tilde{M}$ as a $\Z/n$-module is also free of rank one and we pick a $\Z/n$-Galois covering $\tilde{E}_{K}\to E_{K}$ whose class in $H^{1}(E_{K},\Z/n)$ generates $\tilde{M}$; see Remark \ref{remark:Intuition-for-P1-P2}.\par
We may write $E_{K}\cong\tilde{E}_{K}/\tau$ for some torsion point $\tau\in \tilde{E}_{K}[n]$ of order $n$. Let $k(J)\subset F\subset K$ be a finite intermediate field extension such that the $n$-torsion points of $\tilde{E}_{K}$ are defined over $F$, i.e., $\tilde{E}_{F}[n]\cong (\Z/n)^{\oplus 2}$. Let $C$ be a smooth projective curve over $k$ whose generic point gives the field extension $k\subset F$. We consider the relatively minimal model $p: \mathcal{\tilde{E}}\to C$ of the curve $\tilde{E}_{F}$. In fact, we may choose the field $F$ above, in such a way that $p: \mathcal{\tilde{E}}\to C$ is semi-stable and that every singular fibre has Kodaira type $I_{nm}$ for some integer $m\geq2$; see \cite[VII.5.4]{silverman}. The identity element $s\in \tilde{E}_{F}$ extends to a section $s$ of $p$, which we call the identity section. The $n$-torsion point $\tau\in \tilde{E}_{F}[n]$ also extends to a section of $p$, which we denote again by the same symbol.\par
Recall that the set of reduced components of a fiber in a relatively minimal
elliptic model with a section inherits a group structure from the N\'eron model of
the generic fiber. In the case of a fiber of Kodaira type $I_{N}$ this component
group is isomorphic to $\Z/N$. We label the components $R_{i},$ with $i\in\{0,1,\ldots, N-1\}$, so that the section $s$ meets $R_{0}$ and the intersections $R_{i}\cdot R_{i'}$ are given by 
\begin{itemize}
    \item $R_{i}\cdot R_{i}=-2$, 
    \item $R_{i}\cdot R_{i'}=0$ if  $|i-i'|>1$ and $|i-i'|\neq N-1$,
    \item $R_{i}\cdot R_{i'}=1$ if $|i-i'|=1$ or $|i-i'|=N-1$.
\end{itemize}
Then an isomorphism from the component group to $\Z/N$ is given by $R_{i}\mapsto i\in \Z/N$.\par
We find a point $x_{0}\in C(k)$ such that the fibre $p^{-1}(x_{0})$ has Kodaira type $I_{nm}$ for some $m\geq 2$ and such that the unique component of the fibre $p^{-1}(x_{0})$ meeting $\tau$ has the form $R_{rm}$, with $\gcd(r,n)=1$; see \cite[Lemma 2.7]{schoen-products}. We then observe that the family $\tilde{\mathcal{E}}/\tau\to\Spec\mathcal{O}_{C,x_{0}}$ is a semi-stable model of the elliptic curve $E_{K}$ and has the following properties: Its central fibre has Kodaira type $I_{m},$ with $m\geq2,$ and the $\Z/n$-Galois covering $\tilde{\mathcal{E}}\to \tilde{\mathcal{E}}/\tau$ is trivial over each component of the central fibre.\par
 We can thus apply Theorem \ref{thm:injectivity-CH^d(X,1;Z/n)} for the particular choice of $M$ (see Remark \ref{remark:P'1-P'2} and Remark \ref{remark:Intuition-for-P1-P2}) and obtain that the exterior product maps 
\begin{equation*}
    \Lambda:M\otimes \frac{H^{2c-p-1}_{c-p-2,nr}(Y,\Z/n(c))}{H^{2c-p-1}(Y,\Z/n(c))} \longrightarrow M\otimes \frac{H^{2c-p-1}_{c-p-2,nr}(Y_{K},\Z/n(c))}{H^{2c-p-1}(Y_{K},\Z/n(c))}\overset{\eqref{eq:pairing}}{\longrightarrow} \frac{H^{2c-p}_{c-p-2,nr}(E_{K}\times_{K} Y_{K},\Z/n(1+c))}{H^{2c-p}(E_{K}\times_{K} Y_{K},\Z/n(1+c))}
\end{equation*}
are injective for all $c,p\geq0$. These maps in turn naturally identify with the exterior product maps
\begin{equation}\label{eq:ext-prod-F-b}
 M\otimes F^{c,2c-p}_{n}(Y)\longrightarrow F^{c+1,2c-p+1}_{n}(E_{K}\times_{K}Y_{K}),[\Gamma]\otimes [z]\mapsto [\Gamma\times z_{K}]
\end{equation}
where $F^{c,2c-p}_{n}(Y):=\ker(\cl_{Y}:\CH^{c}(Y,p\ ;\Z/n)\to H^{2c-p}(Y_{\et},\Z/n(c)))$; see Corollary \ref{cor:pairings-comparison}. Finally, we make use of Lemma \ref{lem:kernel-exterior-product} and conclude that the exterior product maps
\eqref{eq:ext-prod-elliptic} are injective for all integers $c,p\geq0$, since the exterior product maps \eqref{eq:ext-prod-F-b} are. This proves Theorem \ref{thm:ext-prod-elliptic}.
\end{proof}

As a consequence of Theorem \ref{thm:ext-prod-elliptic}, we obtain the following result.

\begin{corollary}\label{cor:ext-prod-elliptic} Let $Y$ be a smooth complex projective variety. Then there is a smooth complex elliptic curve $E$ such that the following holds: For each integer $n\geq2,$ and every free $\Z/n$-submodule $M\subset L^{1}H^{1}(E)/n$ of rank one, the exterior product maps
\begin{equation}\label{eq:ext-prod-elliptic-L}
    M\otimes L^{c}H^{2c-p}(Y)/n\longrightarrow L^{c+1}H^{2c+1-p}(E\times Y)/n,\ \alpha\otimes\beta\mapsto [\alpha\times \beta]
\end{equation}
 are injective for all integers $c,p\geq0$.   
\end{corollary}
\begin{proof} We may find a countable algebraically closed field $k\subset\C$ over which $Y$ is defined, i.e., $Y=Y_{0}\times_{k}\C$. We then consider the Legendre elliptic curve $E_{k(t)}$, which is defined over $k(t),$ i.e., the function field of $\mathbb{P}^{1}_{k}$. The inclusion $k\hookrightarrow\C$ factors through $k(t)\hookrightarrow \C$ and we can thus apply Theorem \ref{thm:ext-prod-elliptic} for $K=\C$. That is, for each $n\geq2$ and every free $\Z/n$-submodule $M\subset \CH^{1}(E,1\ ; \Z/n)$ of rank one, we find that the exterior product maps
$$M\otimes \CH^{c}(Y_{0},p\ ;\Z/n)\overset{\eqref{eq:ext-prod-elliptic}}{\longrightarrow} \CH^{c+1}(E\times Y,p+1\ ;\Z/n)$$
are injective for all $c,p\geq0$, where we set $E:=E_{k(t)}\times_{k(t)}\C$. The rigidity theorem of Jannsen \cite[Theorem 0.3]{jannsen-r} implies that the base change map $\CH^{c}(Y_{0},p\ ;\Z/n)\to \CH^{c}(Y,p\ ;\Z/n)$ is an isomorphism. Moreover, note that we have a natural isomorphism $\CH^{1}(E,1\ ;\Z/n)\cong L^{1}H^{1}(E)/n$. This follows from the canonical isomorphism $\CH^{1}(E,1\ ;\Z/n)\cong L^{1}H^{1}(E,\Z/n)$ (see Theorem \ref{thm:suslin-voevodsky} \eqref{it:S-V}) along with the observation that the canonical injection $L^{1}H^{1}(E)/n\hookrightarrow L^{1}H^{1}(E,\Z/n)$ is, in fact, an isomorphism, since $L^{1}H^{2}(E)=\CH^{1}(E)/\sim_{\alg}\cong\Z$ is torsion-free. The comparison theorem between morphic cohomology and higher Chow groups for finite coefficients in turn implies that the exterior product maps $$M\otimes \CH^{c}(Y,p\ ;\Z/n)\overset{\eqref{eq:ext-prod-elliptic}}{\longrightarrow} \CH^{c+1}(E\times Y,p+1\ ;\Z/n)$$
identify with the maps $$M\otimes L^{c}H^{2c-p}(Y,\Z/n)\overset{\times}{\longrightarrow} L^{c+1}H^{2c+1-p}(E\times Y,\Z/n);$$ see Theorem \ref{thm:suslin-voevodsky} \eqref{it:S-V}. Finally, the injectivity of the latter clearly implies the injectivity of the exterior product maps \eqref{eq:ext-prod-elliptic-L}, as we wanted. The proof of Corollary \ref{cor:ext-prod-elliptic} is complete.
\end{proof}
We finally see that Theorem \ref{thm:general-totaro-result} and Corollary \ref{cor:general-totaro-result} are concluded from Corollary \ref{cor:ext-prod-elliptic}, as follows. 
\begin{proof}[Proof of Theorem \ref{thm:general-totaro-result}] We proceed by induction on $p\geq0$. If $p=0$, then $L^{2}H^{4}(X)/n\cong\CH^{2}(X)/n$ and one can take, for instance, $X=JC$, as the Jacobian of a very general genus 3 curve $C$; see \cite[Corollary 0.1]{totaro-annals}. Let $p\geq 1$. The induction hypothesis gives a smooth complex projective variety $Y$ of dimension $p+2$ such that $L^{p+1}H^{p+3}(Y)/n$ contains infinitely many elements of order $n$ for all $n\geq2$. Now, Corollary \ref{cor:ext-prod-elliptic} yields a smooth complex elliptic curve $E$ such that the following holds: For each $n\geq2$, and every free $\Z/n$-submodule $M\subset L^{1}H^{1}(E)/n$ of rank one, the exterior product map 
$$ M\otimes L^{p+1}H^{p+3}(Y)/n\overset{\eqref{eq:ext-prod-elliptic-L}}{\longrightarrow} L^{p+2}H^{p+4}(E\times Y)/n$$
is injective. Thus, for all $n\geq2$, the group $L^{p+2}H^{p+4}(E\times Y)/n$ contains infinitely many elements of order $n$, as the group $L^{p+1}H^{p+3}(Y)/n$ does. The proof of Theorem \ref{thm:general-totaro-result} is complete.\end{proof}

\begin{proof}[Proof of Corollary \ref{cor:general-totaro-result}] Let $p\geq0$ and $d\geq p+3$. Then Theorem \ref{thm:general-totaro-result} yields a smooth complex projective variety $Y$ of dimension $p+3$ such that the group $L^{p+2}H^{p+4}(Y)/n$ contains infinitely many elements of order $n$ for all $n\geq2$. Then the claim follows from the projective bundle formula applied to the $d$-fold $X:=Y\times\mathbb{P}^{d-(p+3)}_{\C}$; see \cite[Proposition 2.5]{FG93}. The proof of Corollary \ref{cor:general-totaro-result} is complete.
\end{proof}

\begin{remark}\label{rem:totaro-example} Let $n\geq 2$ be an integer. It was crucial for the proof of Theorem \ref{thm:general-totaro-result} an example of a smooth complex projective $3$-fold $Y,$ such that $\CH^{2}(Y)/n$ has infinitely many cycles of order $n$. Instead of Totaro's example \cite{totaro-annals}, we could use any of \cite{Diaz,scavia}. In particular, \cite{scavia} yields examples defined over an algebraic closure of $\mathbb{Q}(t)$.
\end{remark}

\subsection{Examples with infinitely many $\ell^{\infty}$-torsion higher Chow cycles modulo a prime $\ell$} In this final subsection, we prove Theorems \ref{thm:main} and \ref{thm:injectivity-result-intro} together with Corollary \ref{cor:main}. Theorem \ref{thm:injectivity-result-intro} will be deduced from the following result. 

\begin{theorem}\label{thm:main-general-field} Let $k$ be an algebraically closed field and let $n\geq2$ be an integer invertible in $k$. Let $k\subset K$ be an algebraically closed field extension of positive transcendence degree. Consider the degeneration $\mathcal{S}\to\Spec\mathcal{O}_{\mathbb{A}^{1}_{k},0},$ which is constructed in Theorem \ref{thm:degeneration}, and let $S_{K}:=\mathcal{S}\times_{\mathcal{O}_{\mathbb{A}^{1}_{k},0}}K$ be its geometric generic fibre. Then for every smooth projective variety $Y$ over $k$, with base change $Y_{K}:=Y\times_{k}K$, the exterior product map  
\begin{equation}\label{eq:ext-prod-S-Z/n}
        \CH^{2}(S_{K},1)/n\otimes \CH^{c}(Y,p\ ;\Z/n)\longrightarrow \CH^{c+2}(S_{K}\times_{K}Y_{K},p+1\ ;\Z/n),\ [\Gamma]\otimes [z]\mapsto [\Gamma\times z_{K}]
\end{equation}
is injective for all $c,p\geq0$.
\end{theorem}
\begin{proof} A rigidity theorem of Jannsen implies that it suffices to consider only the case where the field $K$ is an algebraic closure of $k(t)$; see \cite[Theorem 0.3]{jannsen-r}.\par
We wish to apply Theorem \ref{thm:injectivity-CH^d(X,1)}. To this end, we need to show that the family $\mathcal{S}\to\Spec\mathcal{O}_{\mathbb{A}^{1}_{k},0}$ satisfies the properties \ref{it:P1} and \ref{it:P2}. Note that the geometric generic fibre $S_{K}$ of this family is an \'etale $\Z/n$-quotient of some smooth complete intersection $T_{K}$ of multidegree $(n,n,n)$ within $\mathbb{P}^{5}_{K}$.
The Bockstein map $\delta: H^{1}(S_{K},\Z/n)\to H^{2}_{\cont}(S_{K},\Z_{n})[n]\cong\Z/n$ is an isomorphism, and the generator of the cyclic group $ H^{1}(S_{K},\Z/n)\cong\Z/n$ is given by the $\Z/n$-Galois covering $T_{K}\to T_{K}/(\Z/n)=S_{K}$; see Proposition \ref{pro:properties_S} \eqref{it:Bockstein iso} and Remark \ref{remark:Intuition-for-P1-P2}. This in turn extends to a $\Z/n$-Galois covering $\mathcal{T}\to\mathcal{S},$ that is trivial over each component $S_{0i}$ of the central fibre $S_{0}=\mathcal{S}\times_{{O}_{\mathbb{A}^{1}_{k},0}}k$; see Theorem \ref{thm:degeneration} \eqref{it:galois-covering}. Thus, the class of $\mathcal{T}\to\mathcal{S}$ in $H^{1}(\mathcal{S}_{\et},\Z/n)$ satisfies both \ref{it:P1} and \ref{it:P2}. Note that here once you have \ref{it:P1}, then \ref{it:P2} is automatic, since all components of the central fibre $S_{0}$ are projective planes; see Theorem \ref{thm:degeneration} \eqref{it:S_0}.\par
Recall that $\CH^{2}(S_{K},1)/n\cong\Z/n$ is a free $\Z/n$-module of rank one; see Proposition \ref{pro:properties_S} \eqref{it:CH2(S,1)/n}. It follows from Theorem \ref{thm:injectivity-CH^d(X,1)} that for $M=\CH^{2}(S_{K},1)/n$ and any smooth projective variety $Y$ over $k$, with base change $Y_{K}=Y\times_{k}K,$
the exterior product maps 
\begin{equation*}
    \Lambda:M\otimes \frac{H^{2c-p-1}_{c-p-2,nr}(Y,\Z/n(c))}{H^{2c-p-1}(Y,\Z/n(c))} \longrightarrow M\otimes \frac{H^{2c-p-1}_{c-p-2,nr}(Y_{K},\Z/n(c))}{H^{2c-p-1}(Y_{K},\Z/n(c))}\overset{\eqref{eq:pairing}}{\longrightarrow} \frac{H^{2+2c-p}_{c-p-1,nr}(S_{K}\times_{K} Y_{K},\Z/n(2+c))}{H^{2+2c-p}(S_{K}\times_{K} Y_{K},\Z/n(2+c))}
\end{equation*}
are injective for all $c,p\geq0$. These maps in turn naturally identify with the exterior product maps
\begin{equation}\label{eq:ext-prod-F}
 \CH^{2}(S_{K},1)/n\otimes F^{c,2c-p}_{n}(Y)\longrightarrow F^{c+2,2c-p+3}_{n}(S_{K}\times_{K}Y_{K}),[\Gamma]\otimes [z]\mapsto [\Gamma\times z_{K}]
\end{equation}
where $F^{c,2c-p}_{n}(Y):=\ker(\cl_{Y}:\CH^{c}(Y,p\ ;\Z/n)\to H^{2c-p}(Y_{\et},\Z/n(c)))$; see Corollary \ref{cor:pairings-comparison}. Finally, the injectivity of the cycle class map $\cl_{S_{K}}:\CH^{2}(S_{K},1)/n\hookrightarrow H^{3}({S_{K}}_{\et},\Z/n(2))$ (see Lemma \ref{lem:bound_mod_n} \eqref{it:bound_mod_n}) together with the injectivity of the exterior product maps \eqref{eq:ext-prod-F} imply that the maps 
\begin{equation*}
    \CH^{2}(S_{K},1)/n\otimes\CH^{c}(Y,p\ ;\Z/n)\overset{\eqref{eq:ext-prod-S-Z/n}}{\longrightarrow} \CH^{c+2}(S_{K}\times_{K}Y_{K}, p+1\ ;\Z/n)
\end{equation*}
are also injective for all $c,p\geq0$; see Lemma \ref{lem:kernel-exterior-product}. This concludes the proof of Theorem \ref{thm:main-general-field}.
\end{proof}

\begin{proof}[Proof of Theorem \ref{thm:injectivity-result-intro}]
Let $n\geq 2$ be an integer and let $Y$ be a smooth complex projective variety. Then we can find a countable algebraically closed field $k\subset\C$ such that $Y$ is defined over $k,$ i.e., $Y=Y_{0}\times_{k}\C$. Consider the degeneration $\mathcal{S}\to\Spec \mathcal{O}_{\mathbb{A}^{1}_{k},0}$ constructed in Theorem \ref{thm:degeneration}. Since $k$ is a countable field, the extension $k\hookrightarrow\C$ factors through $k\subset k(t)$. Thus, Theorem \ref{thm:main-general-field} for $K=\C$ implies that the exterior product map   
\begin{equation}\label{eq:ext-prod-Z/n}
        \CH^{2}(S,1)/n\otimes \CH^{c}(Y,p\ ;\Z/n)\longrightarrow \CH^{c+2}(S\times_{\C}Y,p+1\ ;\Z/n)
\end{equation}
is injective for all $c,p\geq0$, where we set $S:=\mathcal{S}\times_{\mathcal{O}_{\mathbb{A}^{1}_{k},0}}\C$. It is important to note that above we have implicitly used that the base change map $\CH^{c}(Y_{0},p\ ;\Z/n)\to \CH^{c}(Y,p\ ;\Z/n)$ is an isomorphism; see \cite[Theorem 0.3]{jannsen-r}. Recall that the surface $S$ satisfies $\Pic(S)_{\tors}=\NS(S)_{\tors}\cong\Z/n$ and thus we obtain a canonical isomorphism $\CH^{2}(S,1)/n\cong L^{2}H^{3}(S)/n\cong\Z/n$; see Lemma \ref{lem:properties} \eqref{it:Pr-b}. Moreover, we find that the exterior product map \eqref{eq:ext-prod-Z/n} identifies with
\begin{equation}\label{eq:ext-prod-L-Z/n}
    L^2H^3(S)/n\otimes L^cH^{2c-p} (Y,\Z/n) \overset{\times}{\longrightarrow} L^{c+2}H^{2c-p+3} (S\times Y,\Z/n)
\end{equation}
for all $c,p\geq0$; see Theorem \ref{thm:suslin-voevodsky} \eqref{it:S-V}. Finally, note that we have a commutative diagram
\begin{equation*}
    \begin{tikzcd}
{L^2H^3(S)/n\otimes L^cH^{2c-p} (Y,\Z/n)} \arrow[r, "\eqref{eq:ext-prod-L-Z/n}"]                       & {L^{c+2}H^{2c-p+3} (S\times Y,\Z/n)}            \\
L^2H^3(S)/n\otimes L^cH^{2c-p} (Y) \arrow[u, hook, shift right] \arrow[r, "\eqref{eq:exterior-product}"] & L^{c+2}H^{2c-p+3} (S\times Y)/n, \arrow[u, hook]
\end{tikzcd}
\end{equation*}
where the vertical maps are injective. Hence, we see that the exterior product map \eqref{eq:exterior-product} is injective, since the exterior product map \eqref{eq:ext-prod-L-Z/n} is. The proof of Theorem \ref{thm:injectivity-result-intro} is complete.\end{proof}
The following is the analogue of Theorem \ref{thm:main} for morphic cohomology.
\begin{theorem}\label{thm:main-Lawson} For each $p\geq1$ and $n\geq2$, there is a smooth complex projective variety $X$ of dimension $p+4$ such that the group $L^{p+3}H^{p+6}(X)[n]/nL^{p+3}H^{p+6}(X)[n^2]$ contains infinitely many elements of order $n$.
\end{theorem}
\begin{proof} Let $p\geq1$ and $n\geq2$. Then Theorem \ref{thm:general-totaro-result} yields a smooth complex projective variety $Y$ of dimension $p+2$ such that the group $L^{p+1}H^{p+3}(Y)/n$ contains infinitely many elements of order $n$ for all integers $n\geq2$. On the other hand, we can find a smooth complex projective surface $S$ with $\Pic(S)_{\tors}=\NS(S)_{\tors}\cong\Z/n$ such that the exterior product map 
$$L^{2}H^{3}(S)/n\otimes L^{p+1}H^{p+3}(Y)/n\overset{\eqref{eq:exterior-product}}{\longrightarrow} L^{p+3}H^{p+6}(S\times Y)/n$$
is injective; see Theorem \ref{thm:injectivity-result-intro}. Recall that we have an equality $L^{2}H^{3}(S)[n]/nL^{2}H^{3}(S)[n^2]=L^{2}H^{3}(S)/n$ and that the latter group is isomorphic to $\Z/n$; see Lemma \ref{lem:properties} \eqref{it:Pr-b}. All together imply that we have an injection 
$$L^{p+1}H^{p+3}(Y)/n \cong L^{2}H^{3}(S)/n\otimes L^{p+1}H^{p+3}(Y)/n\overset{\eqref{eq:exterior-product}}{\longrightarrow} L^{p+3}H^{p+6}(X)[n]/n L^{p+3}H^{p+6}(X)[n^2],$$ where we set $X:=S\times Y$. Hence, the group $L^{p+3}H^{p+6}(X)[n]/n L^{p+3}H^{p+6}(X)[n^2]$ has infinitely many elements of order $n$, as the group $L^{p+1}H^{p+3}(Y)/n$ also has. The proof of Theorem \ref{thm:main-Lawson} is complete.
\end{proof}
We can finally deduce Theorem \ref{thm:main} and Corollary \ref{cor:main}.
\begin{proof}[Proof of Theorem \ref{thm:main}] Let $p\geq1$ and $n\geq2$. Then Theorem \ref{thm:main-Lawson} gives a smooth complex projective variety $X$ of dimension $p+4$ such that the group $L^{p+3}H^{p+6}(X)[n]/n L^{p+3}H^{p+6}(X)[n^2]$ has infinitely many elements of order $n$. Thus, the claim follows from the canonical isomorphism
$$ \CH^{p+3}(X,p)[n]/n\CH^{p+3}(X,p)[n^2]\cong L^{p+3}H^{p+6}(X)[n]/n L^{p+3}H^{p+6}(X)[n^2];$$
see Theorem \ref{thm:suslin-voevodsky} \eqref{it:S-V'}. The proof of Theorem \ref{thm:main} is complete.\end{proof}

\begin{proof}[Proof of Corollary \ref{cor:main}] Let $p\geq1, d\geq p+4$ and $n\geq2$. Then Theorem \ref{thm:main} yields a smooth complex projective variety $Y$ of dimension $p+4$ such that the group $\CH^{p+3}(X,p)[n]/n\CH^{p+3}(X,p)[n^2]$ contains infinitely many cycles of order $n$. Thus, the result follows from the projective bundle formula applied to the $d$-fold $X:=Y\times\mathbb{P}^{d-(p+4)}_{\C}$; see \cite[Theorem 7.1]{bloch-motivic}. The proof of Corollary \ref{cor:main} is complete.\end{proof}

\begin{remark}\label{rem:trivial-abel-jacobi-invariant} Let $n\geq2$. Recall that for $p=1$, the smooth complex projective $5$-fold $X$ of Theorem \ref{thm:main} is given by the product $X:=S\times JC$ of the surface $S$ from Theorem \ref{thm:injectivity-result-intro}, with the Jacobian $JC$ of a very general quartic plane curve $C\subset\mathbb{P}^{2}_{\C}$. In this case the infinitely many cycles considered in $\CH^{4}(X,1)[n]$ all lie in the kernel of the natural map $\cl_{X}:\CH^{4}(X,1)\to H^{7}_{L}(X,\Z(4))$.\par Indeed, the infinitely many cycles in $\CH^{2}(JC)/n\cong L^2H^4(X)/n$ arise from pulling back the Ceresa cycle \cite{ceresa} by one of infinitely many distinct isogenies; see \cite{totaro-annals}. Hence, they are all nullhomologous and in particular lie in the kernel $E^{2}_{n}(JC):=\ker(\bar{\cl}_{JC}:\CH^{2}(JC)/n\to H^{4}(JC_{\et},\Z/n(2)))$. Now, for $[\Gamma]\in\CH^{2}(S,1)[n]$ and $[z]\in E^{2}_{n}(JC)$, one finds that the class $\cl_{S\times JC}([\Gamma\times z])=p^{\ast}\cl_{S}([\Gamma])\cup q^{\ast}\cl_{JC}([z])=0$ is zero, where $p:S\times JC\to S$ and $q:S\times JC\to JC$ are the natural projections. This is because $\cl_{S}([\Gamma])$ is $n$-torsion, while $\cl_{JC}([z])$ is divisible by $n$.
\end{remark}

\begin{remark}
  Let $p\geq1$ and $n\geq2$. Let $X$ be as in the proof of Theorem \ref{thm:main} (or Theorem \ref{thm:main-Lawson}). As in the case $p=1$ (see Remark \ref{rem:trivial-abel-jacobi-invariant}), one sees that the kernel of the cycle class map $\cl^{L}_{X}:\CH^{p+3}(X,p)[n]\to H^{p+6}_{L}(X,\Z(p))[n]$ is infinite, since the $n$-torsion subgroup of $H^{p+6}_{L}(X,\Z(p))$ is finite. However, if $p\geq2$, then our $n$-torsion cycles in $\CH^{p+3}(X,p)$ are no longer chosen canonically, that is, these are lifts of $n$-torsion classes in $L^{p+3}H^{p+6}(X)$. This in turn prevents us from computing their value via the above cycle map $\cl^{L}_{X}$. Instead, if one considers the cycle map $\cl^{B}_{X}:\CH^{p+3}(X,p)\to H^{p+6}_{B}(X,\Z(p))$, with values in Betti cohomology, then we find that the infinitely many $n$-torsion cycles in $\CH^{p+3}(X,p)$ all lie in the kernel of $\cl^{B}_{X}$.\par
  Indeed, the latter cycle class map factors through the $s$-map $L^{p+3}H^{p+6}(X)\to H^{p+6}_{B}(X,\Z)$ constructed in \cite{FM94}. Recall that $X$ is given by the product $S\times Y$, where $S$ is the surface of Theorem \ref{thm:injectivity-result-intro} and where $Y:=E_{1}\times\cdots\times E_{p-1}\times JC$ is the product of $p-1$ very general elliptic curves $E_{i}$ with the Jacobian $JC$ of a very general genus $3$ curve $C$. By our construction the infinitely many $n$-torsion cycles in $L^{p+3}H^{p+6}(X)$ are given as exterior products $[\Gamma\times z]$, where $[\Gamma]\in L^{2}H^{3}(S)/n$ and $[z]\in\ker(L^{p+1}H^{p+3}(Y)/n\to H^{p+3}(Y_{\et},\Z/n(p+1)))$; see proofs of Theorems \ref{thm:main-Lawson} and \ref{thm:general-totaro-result}. Thus, same reasoning as in Remark \ref{rem:trivial-abel-jacobi-invariant}, explains that the $n$-torsion classes considered in $L^{p+3}H^{p+6}(X)$ all lie in the kernel of the $s$-map.\end{remark}

\section*{Acknowledgments} We are grateful to Stefan Schreieder for comments and discussions. The second named author was
supported by the European Union’s Horizon 2020 research and innovation programme under grant agreement No. 948066 (ERC-StG RationAlgic).


\end{document}